\documentclass[reqno,12pt]{amsart}
\usepackage{amssymb}
\usepackage{amsmath}
\usepackage{mathrsfs}
\usepackage{amsmath,amssymb}
\usepackage{paralist}
\usepackage{graphics} 
\usepackage{epsfig} 
\usepackage[colorlinks = true]{hyperref}
\hypersetup{urlcolor = red, citecolor = blue}
\usepackage[top = 1in,bottom = 1in,left = 1in,right = 1in]{geometry}
\usepackage{cite}
\newtheorem{thm}{Theorem}[section]

\newtheorem{lem}[thm]{Lemma}
\newtheorem{prop}[thm]{Proposition}
\theoremstyle{definition}

\newtheorem{rem}{Remark}[section]
\numberwithin{equation}{section}

\begin{document}

\title[ repulsion-consumption Keller-Segel system]
{Boundedness and finite-time blow-up in a repulsion-consumption system with flux limitation}

\author[Zeng]{Ziyue Zeng}%
\address{School of Mathematics, Southeast University, Nanjing 211189, P. R. China}
\email{ziyzzy@163.com}

\author[Li]{Yuxiang Li$^{\star}$}
\thanks{$^{\star}$Corresponding author}

\address{School of Mathematics, Southeast University, Nanjing 211189, P. R. China}
\email{lieyx@seu.edu.cn}

\thanks{Supported in part by National Natural Science Foundation of China (No. 12271092, No. 11671079) and the Jiangsu Provincial Scientific Research Center of Applied Mathematics (No. BK20233002).}



\maketitle
\begin{abstract}
We investigate the following repulsion-consumption system with flux limitation
\begin{align}\tag{$\star$}
  \left\{
  \begin{array}{ll}
   u_t=\Delta u+\nabla \cdot(uf(|\nabla v|^2) \nabla v), & x \in \Omega, t>0, \\
   \tau v_t=\Delta v-u v, & x \in \Omega, t>0,
\end{array}
  \right.
\end{align}
under no-flux/Dirichlet boundary conditions, where $\Omega \subset \mathbb{R}^n$ is a bounded domain and $f(\xi)$ generalizes the prototype given by $f(\xi)=(1+\xi)^{-\alpha}$ ($\xi \geqslant 0$). We are mainly concerned with the global existence and finite time blow-up of system ($\star$). The main results assert that, for $\alpha > \frac{n-2}{2n}$, then when $\tau=1$ and under radial settings, or when $\tau=0$ without radial assumptions, for arbitrary initial data, the problem ($\star$) possesses global bounded classical solutions; for $\alpha<0$, $\tau=0$, $n=2$ and under radial settings, for any initial data, whenever the boundary signal level large enough, the solutions of the corresponding problem blow up in finite time.

Our results can be compared respectively with the blow-up phenomenon obtained by Ahn \& Winkler (2023) for the system with nonlinear diffusion and linear chemotactic sensitivity, and by Wang \& Winkler (2023) for the system with nonlinear diffusion and singular sensitivity .
\end{abstract}

\section{Introduction}
\vskip 3mm
In this paper, we consider the following repulsion-consumption system with flux limitation
\begin{align}\label{1.0.0}
  \left\{
  \begin{array}{ll}
    u_t=\Delta u+\nabla \cdot(uf(|\nabla v|^2) \nabla v), & x \in \Omega, t>0, \\
    \tau v_t=\Delta v-u v, & x \in \Omega, t>0,\\
    \left( \nabla u+uf(|\nabla v|^2)\nabla v \right) \cdot \nu=0, \quad v=M, & x \in \partial \Omega, t>0, \\
    u(x, 0)=u_0(x), \quad \tau v(x, 0)=\tau v_0(x), & x \in \Omega,
\end{array}
  \right.
\end{align}
posed on a bounded domain $\Omega \subset \mathbb{R}^n$, where $\tau \in \{0,1\}$, $M$ is a given parameter and the function $f$ appropriately generalizes the prototype determined by
\begin{align*}
f(\xi)=(1+\xi)^{-\alpha}, \quad \xi \geqslant 0
\end{align*}
with $\alpha \in \mathbb{R}$. The scalar functions $u$ and $v$ represent the cell density and the chemical concentration consumed by cells, respectively. $f(|\nabla v|^2)$ represents the chemotactic sensitivity of cells which is sensitive to gradient of $v$. Detailed biological backgrounds about the Keller-Segel model with flux limitation can be found in \cite{2010-MMMAS-BellomoBellouquidNietoSoler, 2022-MPRIA-Zhigun, 2020-RMI-PerthameVaucheletWang}. In the present work, we aim to find the critical blow-up exponent of flux limitation in system (\ref{1.0.0}).

When $\alpha=0$ and the chemotactic sensitivity function depends on $u$ and $v$, rather than the gradient of $v$, the system is given by
\begin{align}\label{1.0.2}
  \left\{
  \begin{array}{ll}
    u_t=\nabla \cdot(D(u) \nabla u)-\nabla \cdot(uS(u,v) \nabla v),   \\
    \tau v_t=\Delta v-u v, 
 \end{array}
  \right.
\end{align}
where $D(u)$ denotes the diffusivity of the cells and $S(u,v)$ stands the chemotactic sensitivity. The system (\ref{1.0.2}) with $D(u)= 1$ and $S(u,v)= \chi $ is a well-known chemotaxis model which describes the intricate patterns formed by the colonies of Bacillus subtilis as they seek oxygen \cite{1971-JTB-KellerSegel, 1990-PA-MatsushitaFujikawa, 2005-PNASU-TuvalCisnerosDombrowskiWolgemuthKesslerGoldstein}. In the past decades, system (\ref{1.0.2}) with $\tau=1$, subjected to homogeneous Neumann boundary conditions, has been studied extensively on the existence of global bounded solutions. When $\left\|v_0\right\|_{L^{\infty}(\Omega)}$ is sufficiently small and $n\geqslant 2$, Tao \cite{2011-JMAA-Tao} proved that system (\ref{1.0.2}) possesses global bounded classical solutions. For arbitrary large initial data, Tao and Winkler \cite{2012-JDE-TaoWinklera} demonstrated that if $n=2$, the global classical solutions of (\ref{1.0.2}) are bounded; if $n=3$, there is at least one global weak solution, which eventually becomes bounded and smooth. In \cite{2019-EJDE-WangLi}, Wang and Li showed that this model possesses at least one global renormalized solution when $n \geqslant 4$. We refer to \cite{2015-ZAMP-WangMuLinZhao, 2014-ZAMP-WangMuZhou, 2015-ZAMP-WangXiang} and survey \cite{2023-SAM-LankeitWinkler} for more related results.

The system (\ref{1.0.2}), subjected to no-flux/Dirichlet boundary conditions, has been investigated by some authors. 
For the parabolic-parabolic system (\ref{1.0.2}) with $D(u)=S(u,v)=1$, Lankeit and Winkler \cite{2022-N-LankeitWinkler} found that the radially symmetric problem possesses global bounded classical solutions when $n=2$ and global weak solutions when $n\in \left\{3,4,5\right\}$. For the parabolic-elliptic system (\ref{1.0.2}) with $D(u)=1$ and $S(u,v)=S(v)$, where $S(v)$ may allow singularities at $v= 0$, Yang and Ahn \cite{2024-NARWA-YangAhn} established the global existence and boundedness of radial large data solutions when $n \geqslant 2$. The system (\ref{1.0.2}) with $S(u,v)<0$ is called the repulsion-consumption system, and there are some results on the occurrence of blow-up in finite time for the parabolic-elliptic radially symmetric system. For system (\ref{1.0.2}) with $D(u)=(1+u)^{-\alpha}$ and $S(u,v)=-1$, Ahn and Winkler \cite{2023-CVPDE-AhnWinkler} demonstrated that when $\alpha>0$ and the boundary signal level is sufficiently large, the corresponding 2D problem admits a finite-time blow-up classical solution; when $\alpha \leqslant 0$, a global bounded classical solution exists. For system (\ref{1.0.2}) with $D(u)=1$, $S(u,v)=-u^{\beta}$, under radial settings, Zeng and Li \cite{2024--ZengLi} proved that, when $0<\beta < \frac{n+2}{2n}$ and $\tau \in \{0,1\}$, the system possesses global bounded classical solutions; when $\beta>1$, $\tau=0$ and  $n=2$, under the condition that the boundary signal level large enough, there exists a finite-time blow-up solution. For $D(u)=(1+u)^{-\alpha}$ and the chemotactic term in (\ref{1.0.2}) is replaced by $+\nabla \cdot(u \nabla \log v)$, Wang and Winkler \cite{2023-PRSESA-WangWinkler} showed that when $\alpha>0$, for all initial data from a considerably large set of radial functions, the system (\ref{1.0.2}) possess a finite-time blow-up solution.

Bellomo and Winkler \cite{2017-CPDE-BellomoWinkler} considered the following degenerate chemotaxis system with flux limitation
\begin{align}\label{1.0.3}
  \left\{
  \begin{array}{ll}
    u_t = \nabla\cdot\Big( \frac{u\nabla u}{\sqrt{u^2+|\nabla u|^2}}\Big)
          - \chi\nabla\cdot\Big(\frac{u\nabla v}{\sqrt{1+|\nabla v|^2}}\Big), \\
    0 = \Delta v-\mu+u,  
  \end{array}
  \right.
\end{align}
under no-flux/Dirichlet boundary conditions, and they proved that when either $\chi<1$ $(n \geqslant 2)$, or $\int_{\Omega} u_0<\frac{1}{\sqrt{\left(\chi^2-1\right)_{+}}}$ $(n=1)$, solutions are global and bounded. Later, they \cite{2017-TAMS-BellomoWinkler} showed that when $\chi>1$ $(n \geqslant 2)$ or $\int_{\Omega} u_0 \mathrm{~d} x>\frac{1}{\sqrt{\chi^2-1}}$ $(n=1)$, there exist positive initial data $u_0$ such that the corresponding solutions blow up in finite time. Some systems related to (\ref{1.0.3}) were investigated recent years, see \cite{2021-CPAA-YiMuQiuXu, 2019-JDE-MizukamiOnoYokota, 2020-AAM-ChiyodaMizukamiYokota, 2022-MMMAS-TuMuZheng, 2024-MMMAS-LiYan}.

In 2018, Negreanu and Tello \cite{2018-JDE-NegreanuTello} considered the system 
\begin{align}\label{1.0.4}
  \left\{
  \begin{array}{ll}
    u_t =\Delta u-\nabla\cdot(\chi u|\nabla v|^{p-2}\nabla v) + g(u), \\
    0=\Delta v-\mu+u,   
  \end{array}
  \right.
\end{align}
subjected to homogeneous Neumann boundary conditions, and obtained global bounded solutions when $p \in\big(1, \frac{n}{(n-1)_{+}}\big)$ ($n \geqslant 1$) and $g(u)=0$. Later, Tello \cite{2022-CPDE-Tello} demonstrated that there exist initial data satisfying $\frac{1}{|\Omega|} \int_{\Omega} u_0 d x>6$ such that the radially symmetric solutions of system (\ref{1.0.4}) with $g(u)=0$ blow up in finite time when $p \in (\frac{n}{n-1},2)$ ($n>2$) and $\chi$ large enough. Kohatsu \cite{2024-AAM-Kohatsu} proved that system (\ref{1.0.4}) with $g(u)=\lambda u-\mu u^\kappa$, when $p > 1$ $($n = 1$)$ or $p\in(1,\frac{n}{n-1})$ $(n \geqslant 2)$, for all initial data, the problem admits a global bounded weak solution; when $p>\frac{n}{n-1}$ $(n \geqslant 2)$ and $k>1$ is small enough, the radially symmetric problem has a finite-time blow-up weak solution.

The system (\ref{1.0.4}) with the chemotactic term replaced by $-\nabla \cdot (u(1+|\nabla v|^2)^{-\alpha} \nabla v)$ as follows
\begin{align}\label{1.0.5}
  \left\{
  \begin{array}{ll}
    u_t =\Delta u-\nabla \cdot (u(1+|\nabla v|^2)^{-\alpha} \nabla v),  \\
     0=\Delta v-\mu+u,  
  \end{array}
  \right.
\end{align}
has attracted interest by some mathematicians. Winkler \cite{2022-IUMJ-Winkler} proved that when $0<\alpha<\frac{n-2}{2(n-1)}$ ($n \geqslant 3$), throughout a considerably large set of radially symmetric initial data, the problem (\ref{1.0.5}) admits radially symmetric solutions blowing up in finite time with respect to the $L^{\infty}$ norm of $u$; when $\alpha>\frac{n-2}{2(n-1)}$ ($n \geqslant 2$) or $\alpha \in \mathbb{R}$ ($n=1$), the problem exists global bounded classical solutions without radial assumptions. Subsequently, Marras, Vernier-Piro and Yokota \cite{2022-JMAA-MarrasVernierPiroYokota} demonstrated that a solution for system (\ref{1.0.5}) which blows up in finite time in $L^{\infty}$-norm, blows up also in $L^{p}$-norm for $\frac{n}{2}<p<n$ ($n \geqslant 3$) under the same conditions as those in \cite{2022-IUMJ-Winkler} regarding $f$ and initial data. Moreover, they derived a lower bound of blow-up time. For system with source term and $\alpha \in(0,\frac{n-2}{2(n-1)})$ ($n \geqslant 3$), Marras, Vernier-Piro and Yokota \cite{2023-NNDEA-MarrasVernier-PiroYokota} illustrated that weak logistic dampening cannot prevent the occurrence of finite-time blow-up phenomenon in an appropriate and explicit sense. Furthermore, Mao and Li \cite{2024-NA-MaoLi} considered the instability of large homogeneous steady state when $\alpha \in (0,\frac{n-2}{2(n-1)}]$. 

For the parabolic-parabolic system,
\begin{align}\label{1.0.6}
\left\{
  \begin{array}{ll}
u_t=\Delta u-\nabla \cdot(u f(|\nabla v|) \nabla v),  \\ 
v_t=\Delta v-v+u, 
 \end{array}
  \right.
\end{align}
Yan and Li \cite{2020-EJDE-YanLi} proved that the system possesses global weak solutions which are uniformly bounded when $f(\xi)=\xi^{\alpha-2}$ with $\alpha \in (1,\frac{n}{(n-1)_+})$. Kohatsu and Yokota \cite{2023-MC-KohatsuYokota} revealed the
stability of constant equilibria under some smallness conditions for the initial data. Winkler \cite{2022-MN-Winkler} demonstrated that when $f(\xi)=(1+\xi^2)^{-\alpha}$ with $\alpha>\frac{n-2}{2(n-1)}$ ($n \geqslant 2$) or $\alpha \in \mathbb{R}$ ($n=1$), system (\ref{1.0.6}) admits a unique global bounded classical solution. When $f(\xi)=\xi^{\alpha-2}$ and the second equation is replaced by $v_t=\Delta v-u v$, Wang and Li \cite{2019-JMP-WangLi} showed that if $\left\|v_0\right\|_{L^{\infty}(\Omega)}$ is small enough, the model admits at least one global weak solution for $n<\frac{8-2(\alpha-1)}{\alpha-1}$ and possesses at least one global renormalized solution for $n \geqslant \frac{8-2(\alpha-1)}{\alpha-1}$.



In view of \cite{2023-CVPDE-AhnWinkler, 2023-PRSESA-WangWinkler, 2024--ZengLi}, the blow-up phenomenon of the repulsion-consumption system may occur when diffusion is inhibited or the chemotactic sensitivity is sufficiently strong. Motivated by these observations, we investigate the effect of the flux limitation on the occurrence of blow-up phenomenon for the repulsion-consumption parabolic-elliptic system and establish the boundeness of solutions for the repulsion-consumption system to find the critical exponent. 

\textbf{Main results}.
Suppose that $f$ and the initial data satisfy
\begin{align}\label{f-lip}
  f \in C^2([0,\infty))
\end{align}
and 
\begin{align}\label{u_0}
  u_0 \in W^{1,\infty}(\bar{\Omega}) \text { is nonnegative with } u_0 \not \equiv 0.
\end{align}
When $\tau=1$, we assume that 
\begin{align}\label{v 0}
v_0 \in W^{1, \infty}(\Omega) \text { is positive in } \bar{\Omega} \text { and radially symmetric with } v_0=M \text { on } \partial \Omega .
\end{align}

Before we state our main results, we give the local existence of classical solutions to (\ref{1.0.0}), which can be proved by a direct adaptation of well-known fixed-point
arguments \cite{2019-JDE-TaoWinkler}.
 
\begin{prop}\label{0.0-1}
Let $\Omega \subset \mathbb{R}^n$ be a bounded domain with smooth boundary. Suppose that  $(\ref{f-lip})$, $(\ref{u_0})$ and $(\ref{v 0})$ are valid. Then there exist $T_{\max } \in(0, \infty]$ and a uniquely pair $(u, v)$ which solves $(\ref{1.0.0})$ with $\tau=1$ classically in $\Omega \times\left(0, T_{\max }\right)$, satisfying 
\begin{align*}
\left\{\begin{array}{l}
 u \in C(\bar{\Omega} \times[0, T_{\max})) \cap C^{2,1}(\bar{\Omega} \times(0, T_{\max})) \quad \text { and } \\
 v \in \bigcap_{q>n} C\left([0, T_{\max}) ; W^{1, q}(\Omega)\right) \cap C^{2,1}(\bar{\Omega} \times(0, T_{\max}))
\end{array}\right.
\end{align*}
and $u, v > 0$ in $\Omega \times\left(0, T_ {\max }\right)$. Besides, we have
\begin{align}\label{mass-1}
\int_{\Omega} u(\cdot, t)=\int_{\Omega} u_0 \quad \text { for all } t \in\left(0, T_{\max }\right)
\end{align}
and
\begin{align*}
\text { if } T_{\max }<\infty, \quad \text { then } \quad \limsup _{t \nearrow T_{\max }}\|u(\cdot, t)\|_{L^{\infty}(\Omega)}=\infty 
\end{align*}
as well as
\begin{align}\label{vmax}
  \|v(\cdot, t)\|_{L^{\infty}(\Omega)} \leqslant \|v_0\|_{L^{\infty}(\Omega)},
   \quad t \in (0,T_{\max}).
\end{align}
Moreover, if $\Omega=B_R(0)$ with some $R>0$, $u_0$ and $v_0$ are radially symmetric with respect to the origin, then $u(\cdot, t)$ and $v(\cdot, t)$ are radially symmetric for each $t \in\left(0, T_{\max }\right)$.
\end{prop}
\begin{prop}\label{0.0-0}
Let $\Omega \subset \mathbb{R}^n$ be a bounded domain with smooth boundary. Suppose that $(\ref{f-lip})$ and $(\ref{u_0})$ are valid. Then there exist $T_{\max } \in(0, \infty]$ and a unique pair $(u, v)$ which solves $(\ref{1.0.0})$ with $\tau=0$ classically in $\Omega \times\left(0, T_{\max }\right)$, satisfying 
\begin{align*}
\left\{\begin{array}{l}
u \in C\left(\bar{\Omega} \times\left[0, T_{\max }\right)\right) \cap C^{2,1}\left(\bar{\Omega} \times\left(0, T_{\max }\right)\right), \\
v \in \bigcap_{q>n} L_{\mathrm{loc}}^{\infty}\left(\left[0, T_{\max }\right) ; W^{1, q}(\Omega)\right) \cap C^{2,0}\left(\bar{\Omega} \times\left(0, T_{\max }\right)\right)
\end{array}\right.
\end{align*}
and $u, v > 0$ in $\Omega \times\left(0, T_{\text {max }}\right)$. Besides, we have
\begin{align}\label{mass-0}
\int_{\Omega} u(\cdot, t)=\int_{\Omega} u_0 \quad \text { for all } t \in\left(0, T_{\max }\right)
\end{align}
and 
\begin{align*}
\text { if } T_{\max }<\infty, \quad \text { then } \quad \limsup _{t \nearrow T_{\max }}\|u(\cdot, t)\|_{L^{\infty}(\Omega)}=\infty 
\end{align*}
as well as
\begin{align*}
v \leqslant M,
\quad  (x,t) \in \Omega \times \left(0, T_{\max }\right).
\end{align*}
Moreover, if $\Omega=B_R(0)$ with some $R>0$ and $u_0$ is radially symmetric with respect to the origin, then $u(\cdot, t)$ and $v(\cdot, t)$ are radially symmetric for each $t \in\left(0, T_{\max }\right)$.
\end{prop}

Now we state our main results. The existence of global bounded solutions for system (\ref{1.0.0}) with $\tau \in \left\{0,1 \right\}$ is stated as follows:
\begin{thm}\label{0.3}
Let $\Omega=B_{R}(0) \subset \mathbb{R}^n$ $(n \geqslant 2)$ with $R>0$. Suppose that $f(\xi)$ satisfies $(\ref{f-lip})$ and
\begin{align}\label{f-global-1}
  f(\xi) \leqslant  K_{1}(1+\xi)^{-\alpha}, 
  \quad \xi \geqslant 0
\end{align}
with some $\alpha > \frac{n-2}{2n}$ and $K_{1}>0$. Then for any radially symmetric $u_0$ and $v_0$ complying with $(\ref{u_0})$ and $(\ref{v 0})$ respectively, the radially symmetric solution $(u,v)$ of $(\ref{1.0.0})$ with $\tau=1$ is global and bounded in the sense that
\begin{align*}
  \|u(\cdot, t)\|_{L^{\infty}\left(\Omega \right)} \leqslant C, 
  \quad  t>0
\end{align*}
with some $C>0$ depending on $\alpha$, $K_{1}$ and $R$.
\end{thm}
\begin{thm}\label{0.2}
Let $\Omega \subset \mathrm{R}^n$ $(n \geqslant 2)$ be a bounded domain with smooth boundary. Suppose that $f(\xi)$ satisfies $(\ref{f-lip})$ and
\begin{align}\label{f-global}
  f(\xi)\leqslant K_{0} (1+\xi)^{-\alpha}, 
  \quad \xi \geqslant 0
\end{align}
with some $\alpha > \frac{n-2}{2n}$ and $K_{0}>0$. Then for any $u_0$ complying with $(\ref{u_0})$, the solution $(u,v)$ of $(\ref{1.0.0})$ with $\tau=0$ is global and bounded in the sense that
\begin{align*}
  \|u(\cdot, t)\|_{L^{\infty}\left(\Omega \right)} \leqslant C, 
  \quad  t>0
\end{align*}
with some $C>0$ depending on $\alpha$ and $K_{0}$.
\end{thm}
The following theorem is concerned with the finite-time blow-up solution of the 2D system (\ref{1.0.0}) with $\tau=0$, which indicates that $\alpha=0$ is the optimal critical exponent for the occurrence of blow-up phenomenon in the 2D parabolic-elliptic problem:
\begin{thm}\label{0.1}
Let $\Omega=B_R(0) \subset \mathbb{R}^2$ with $R>0$. Suppose that $f(\xi)$ satisfies $(\ref{f-lip})$ and
\begin{align}\label{f-blow-up}
  f(\xi)\geqslant k (1+\xi)^{-\alpha}, 
  \quad \xi \geqslant 0
\end{align}
with some $ \alpha<0 $ and $k>0$. Then, for any radially symmetric $u_0$ complying with $(\ref{u_0})$, there exists a constant $M^{\star}=M^{\star}\left(u_0\right)>0$, whenever $M\geqslant M^{\star}$, the corresponding classical solution $(u, v)$ of $(\ref{1.0.0})$ with $\tau=0$ blows up in finite time.
\end{thm}
\begin{rem}
Owing to the limitation of the method, the question of whether finite-time blow-up occurs in the parameter region $\alpha<\frac{n-2}{2n}$ remains unresolved.
\end{rem}
The rest of the paper is organized as follows. In section~\ref{section 3}, we prove the existence of global bounded solutions for system  (\ref{1.0.0}) with $\tau=1$ and $\alpha >\frac{n-2}{2n}$ under radial settings. In section~\ref{section 2}, we obtain global bounded classical solutions for system (\ref{1.0.0}) with $\tau=0$ and $\alpha >\frac{n-2}{2n}$ without any radial assumptions. In section~\ref{section 1}, we demonstrate that the solution of the 2D radially symmetric problem (\ref{1.0.0}) with $\tau=0$ and $\alpha<0$ blows up in finite time when the boundary signal level sufficiently large. 
\section{Boundedness when $\tau=1$ and $\alpha>\frac{n-2}{2n}$}\label{section 3}
In this section, we assume that $\Omega=B_R(0) \subset \mathbb{R}^n$ $(n \geqslant 2)$ and $\tau=1$. The purpose of this section is to derive the boundedness result stated in Theorem~\ref{0.3} under radial settings. The estimates of boundary terms are crucial in this section because $v$ satisfies the Dirichlet boundary conditions. Inspired by \cite{2022-N-LankeitWinkler}, we use the radial symmetry of the domain to consider the properties of the solution near the boundary, as listed in Lemma~\ref{uR-time-all} and Lemma~\ref{vr-all}. These properties help us to deal with the boundary integrals proposed in the study of the differential inequality.

We collect two lemmas which will be used later.

\begin{lem}
(\cite[Lemma 3.1]{2012-JDE-TaoWinkler}) Let $\beta>0$ and $\gamma>0$ be such that $\beta+\gamma<1$. Then for all $\varepsilon>0$ there exists $C_1(\varepsilon)>0$ such that 
\begin{align}\label{a-b}
a^\beta b^\gamma \leqslant \varepsilon(a+b)+C_1(\varepsilon),
 \quad  \text{for all } \ a \geqslant 0 \ \text{and} \ b \geqslant 0.
\end{align}
\end{lem}
\begin{lem}\label{lem-3.1.0.2}
(\cite[Lemma 3.4]{2019-JFA-Winkler}) Let $t_0 \in \mathbb{R}, T \in\left(t_0, \infty\right]$, $h>0$ and $b>0$. Assume that the nonnegative function $g \in$ $L_{\mathrm{loc }}^1(\mathbb{R})$ satisfies
\begin{align*}
\frac{1}{h} \int_t^{t+h} g(s) d s \leqslant b, \quad t \in\left(t_0, T\right) \text {. }
\end{align*}
Then for any $a>0$ we have
\begin{align*}
\int_{t_0}^t e^{-a(t-s)} g(s) d s \leqslant \frac{b h}{1-e^{-a h}}, \quad t \in\left[t_0, T\right) \text {. }
\end{align*}
Consequently, if $y \in C^0\left(\left[t_0, T\right)\right) \cap C^1\left(\left(t_0, T\right)\right)$ satisfies
\begin{align*}
y^{\prime}(t)+a y(t) \leqslant g(t), \quad t \in\left(t_0, T\right),
\end{align*}
then
\begin{align*}
y(t) \leqslant y\left(t_0\right)+\frac{b h}{1-e^{-a h}}, \quad t \in\left[t_0, T\right).
\end{align*}
\end{lem}

 We provide a basic observation for $v$ as follows:
\begin{lem}\label{heat semigroup}
Let $(u,v)$ be the classical solution of problem $(\ref{1.0.0})$ with $\tau=1$. Suppose that $(\ref{f-lip})$, $(\ref{u_0})$ and $(\ref{v 0})$ are valid. Then, for each $a \in\left[1, \frac{n}{n-1}\right)$, $\sigma \in (1,+\infty)$ and $r_0 \in (0,R)$, there exist constants $C_2=C_2(a)>0$ and $C_3=C_3(\sigma,r_0)>0$ such that
\begin{align}\label{nablavs}
  \|\nabla v(\cdot, t)\|_{L^{ a}(\Omega)} \leqslant C_2,
  \quad t \in (0,T_{\max})
\end{align}
and
\begin{align}\label{vr semilinear estimate}
  \| v_r(\cdot, t)\|_{L^{ \sigma}(r_0,R)} \leqslant C_3,
  \quad t \in (0,T_{\max}).
\end{align}
\end{lem}
\begin{proof}
For $1 \leqslant {p_2} \leqslant p_1 \leqslant \infty$, due to the known smoothing properties of the Dirichlet heat semigroup $\left(e^{t \Delta}\right)_{t \geqslant 0}$ on $\Omega$  ( \cite{1970-TMMO-EidelmanIvasishen} \cite[section 48.2]{2019-QuittnerSouplet}), we can find positive constants $\lambda$, $C_{4}$ and $C_{5}$ such that
\begin{align}\label{Delta-u_0}
\left\|\nabla e^{t \Delta} \varphi\right\|_{L^{p_1}(\Omega)} 
\leqslant C_4\|\varphi\|_{W^{1, \infty}(\Omega)} ,
\quad  \varphi \in W^{1, \infty}(\Omega) \text{ and } \varphi =0 \text{ on } \partial \Omega
\end{align}
and
\begin{align}\label{Delta-u}
\left\|\nabla e^{t \Delta} \varphi\right\|_{L^{p_1}(\Omega)} 
\leqslant C_5 \cdot\left(1+t^{-\frac{1}{2}-\frac{n}{2}(\frac{1}{{p_2}}-\frac{1}{p_1})}\right) e^{-\lambda t}\|\varphi\|_{L^{p_2}(\Omega)}, 
\quad  \varphi \in C(\bar{\Omega}) \text{ and } \varphi =0 \text{ on } \partial \Omega.
\end{align}
For $a \in [1,\frac{n}{n-1})$, by means of the Duhamel representation of $v$, along with (\ref{Delta-u_0}) and (\ref{Delta-u}), we deduce that 
\begin{align*}
\left\|\nabla v(\cdot, t)\right\|_{L^a(\Omega)} 
= &\|\nabla\left(v(\cdot, t)-M\right)\|_{L^a(\Omega)} \notag \\
= &\left\|\nabla e^{t \Delta}\left(v_0-M\right)-\int_0^t \nabla e^{(t-s) \Delta}\Big(u(\cdot, s) v(\cdot, s)\Big) \mathrm{d} s\right\|_{L^a(\Omega)} \notag \\
\leqslant & C_4\left\|v_0-M\right\|_{W^{1, \infty}(\Omega)} \notag \\
& +C_5 \|v_0\|_{L^{\infty}(\Omega)}  \|u_0\|_{L^{1}(\Omega)} \int_0^t\left(1+(t-s)^{-\frac{1}{2}-\frac{n}{2}+\frac{n}{2a}}\right) e^{-\lambda(t-s)} \mathrm{d} s \notag \\
\leqslant & C_4\left\|v_0-M\right\|_{W^{1, \infty}(\Omega)} \notag \\
&+C_5 \|v_0\|_{L^{\infty}(\Omega)} \|u_0\|_{L^{1}(\Omega)} \int_0^{\infty}\left(1+t^{-\frac{1}{2}-\frac{n}{2}+\frac{n}{2a}}\right) e^{-\lambda t} \mathrm{d} t,
\quad t \in (0,T_{\max}),
\end{align*}
which implies (\ref{nablavs}).

To prepare for deriving local estimates, we choose a cut off function $\chi(r) \in C^{\infty}([0,R])$ such that $0 \leqslant \chi \leqslant 1$ fulfilling $\chi \equiv 0$ in $\left[0, \frac{{r_0}}{2}\right]$, and $\chi \equiv 1$ in $[{r_0}, R]$. According to the second equation in (\ref{1.0.0}) with $\tau=1$, we have
\begin{align}\label{v}
\big(\chi(r) \left(v(r,t)-M\right)\big)_t=\big(\chi(r)\left(v(r,t)-M\right)\big)_{r r}+b(r, t),
 \quad (r,t) \in(0, R) \times \left(0, T_{\max}\right),
\end{align}
where
\begin{align}\label{b}
b(r, t) = & \left(\frac{n-1}{r} \chi(r)-2 \chi_r(r)\right) v_{r}(r, t)-\chi_{r r}(r) \big(v(r, t)-M\big) \notag \\
& -\chi(r)  u(r, t) v(r, t),
 \quad (r,t) \in(0, R) \times \left(0, T_{\max }\right) .
\end{align}
By means of (\ref{mass-1}) and (\ref{nablavs}), we can deduce that 
\begin{align}\label{b-1}
b(r,t)\in L^{\infty}\left((0,T_{\max});L^1\big(\frac{{r_0}}{2},R\big)\right).
\end{align}
For all ${\sigma} \in (1, \infty)$, We then again employ the standard smoothing estimates for the Dirichlet heat semigroup $\left(e^{t \Delta}\right)_{t \geqslant 0}$ on $\Omega$ in one dimension as in Lemma~\ref{heat semigroup}, and use (\ref{v}), (\ref{b}) and (\ref{b-1}), to fix $C_3=C_3({r_0},{\sigma})>0$ such that
\begin{align}\label{vr semilinear estimate-1}
& \left\| v_r\right\|_{L^{\sigma}\left(({r_0}, R)\right)} 
 \leqslant \left\|\partial_r\big(\chi \cdot\left(v(\cdot, t)-M\right)\big)\right\|_{L^{\sigma}\left((\frac{{r_0}}{2}, R)\right)} \notag \\ 
= &\left\|\partial_r e^{t \Delta}\big(\chi \cdot\left(v_0-M\right)\big)+\int_0^t \partial_r e^{(t-s) \Delta} b(\cdot, s) \mathrm{d} s\right\|_{L^{\sigma}\left((\frac{{r_0}}{2}, R)\right)} \notag \\
\leqslant & C_4\left\|\chi \cdot\left(v_0-M\right)\right\|_{W^{1, \infty}\left((\frac{{r_0}}{2}, R)\right)} \notag \\
           & +C_5 \int_0^t\left(1+(t-s)^{-1+\frac{1}{2{\sigma}}}\right) e^{-\lambda(t-s)}\left\|b(\cdot, s)\right\|_{L^1\left((\frac{{r_0}}{2}, R)\right)} \mathrm{d} s \notag \\
\leqslant & C_3,
  \quad t \in (0,T_{\max}),
\end{align}
which implies (\ref{vr semilinear estimate}).
\end{proof}

The following Lemma presents an estimate for the integral of $u(R,t)$ with respect to $t$.
\begin{lem}\label{uR-time-all}
Let $(u,v)$ be the classical solution of problem $(\ref{1.0.0})$ with $\tau=1$. Suppose that $f(\xi)$ satisfies $(\ref{f-lip})$ and $(\ref{f-global-1})$ with $\alpha \in (\frac{n-2}{2n},\frac{1}{2})$ and $K_{1}>0$. Then, for any radially symmetric $u_0$ and $v_0$ satisfy $(\ref{u_0})$ and $(\ref{v 0})$ respectively, there exists a constant $C_6=C_6(\alpha,R,K_{1})$ such that 
\begin{align}\label{uR}
\int_t^{t+h} u(R, \sigma) \mathrm{d} \sigma \leqslant C_6,
\quad t \in (0,T_{\max}),
\end{align}
where $h=\min \left\{1, \frac{1}{2} T_{\max}\right\}$.
\end{lem}
\begin{proof}
We choose $\zeta \in C^{\infty}(\bar{\Omega})$ such that $0 \leqslant \zeta \leqslant 1$ fulfilling $\zeta \equiv 0$ in $\bar{B}_{\frac{R}{4}}(0)$ and $\zeta \equiv 1$ in $\bar{\Omega} \backslash \bar{B}_{\frac{R}{2}(0)}$. 
Testing the first equation in (\ref{1.0.0}) by $\zeta^2u^{{q}-1}$ $({q} \in (0,1))$ yields that
\begin{align}\label{zeta2up}
  \frac{1}{{q}} \frac{\mathrm{d}}{\mathrm{d} t} \int_{\Omega} \zeta^2 u^{q} \mathrm{d}x
  = & \int_{\Omega} \zeta^2 u^{{q}-1} (\Delta u + \nabla \cdot(uf(| \nabla v|^2)\nabla v) \mathrm{d}x \notag \\
= & (1-{q}) \int_{\Omega} \zeta^2 u^{{q}-2}\left|\nabla u\right|^2 \mathrm{d}x
    +(1-{q}) \int_{\Omega} \zeta^2 u^{{q}-1}f(| \nabla v|^2) \nabla u \cdot \nabla v \mathrm{d}x \notag \\
& -2 \int_{\Omega} \zeta u^{{q}-1} \nabla u \cdot \nabla \zeta \mathrm{d}x
    -2 \int_{\Omega} \zeta u^{q} f(| \nabla v|^2) \nabla v \cdot \nabla \zeta \mathrm{d}x.
\end{align}
for all $t \in (0,T_{\max})$. 

Next, we show that the first term on the right-hand side of (\ref{zeta2up}) is spatio-temporally integrable. Applying Young's inequality and (\ref{f-global-1}), we have 
\begin{align}\label{nablau nablav}
& \left| (1-{q}) \int_{\Omega} \zeta^2 u^{{q}-1}f(| \nabla v|^2) \nabla u \cdot \nabla v \mathrm{d}x\right| \notag \\
\leqslant &  \frac{(1-{q})}{4} \int_{\Omega} \zeta^2 u^{{q}-2}\left|\nabla u\right|^2 \mathrm{d}x 
          + (1-q)K^2_{1} \int_{\Omega} {\zeta}^2 \frac{u^{q}}{(1+| \nabla v|^2)^{2\alpha}} |\nabla v|^2 \mathrm{d}x \notag \\
\leqslant &  \frac{(1-{q})}{4} \int_{\Omega} \zeta^2 u^{{q}-2}\left|\nabla u\right|^2 \mathrm{d}x 
          + (1-q)K^2_{1} \int_{\Omega} {\zeta}^2 u^{q} |\nabla v|^{2-4\alpha}  \mathrm{d}x
\end{align}
and
\begin{align}\label{nablau nablazeta-1}
\left| -2 \int_{\Omega} \zeta u^{{q}-1} \nabla u \cdot \nabla \zeta \mathrm{d}x \right|
\leqslant \frac{1-{q}}{4}  \int_{\Omega} \zeta^2 u^{{q}-2}\left|\nabla u\right|^2 \mathrm{d}x 
           + \frac{4}{(1-{q})} \int_{\Omega} u^{q} |\nabla \zeta|^2  \mathrm{d}x ,
\end{align}
as well as
\begin{align}\label{nablav nablazeta}
\left| -2 \int_{\Omega} \zeta u^{q} f(| \nabla v|^2) \nabla v \cdot \nabla \zeta \mathrm{d}x \right|
 \leqslant & K^2_{1}\int_{\Omega} {\zeta}^2 \frac{u^{q}}{(1+| \nabla v|^2)^{2\alpha}} |\nabla v|^2 \mathrm{d}x
            + \int_{\Omega} u^{q} |\nabla \zeta|^2  \mathrm{d}x \notag \\
 \leqslant & K^2_{1}\int_{\Omega} {\zeta}^2 u^{q} |\nabla v|^{2-4\alpha} \mathrm{d}x
            + \int_{\Omega} u^{q} |\nabla \zeta|^2  \mathrm{d}x .
\end{align}
For the same terms on the right side of (\ref{nablau nablav})-(\ref{nablav nablazeta}), we use Hölder's inequality with $0<{q}<1$ to obtain
\begin{align}\label{up nabla zeta2}
\int_{\Omega} u^{q} |\nabla \zeta|^2  \mathrm{d}x
  \leqslant  \Big(\int_{\Omega} u\Big)^{q} 
             \cdot \Big(\int_{\Omega}|\nabla \zeta|^{\frac{2}{1-{q}}} \mathrm{d}x \Big)^{1-{q}}
\end{align}
and
\begin{align}\label{zeta2 up nablav 2-4alpha}
\int_{\Omega} {\zeta}^2 u^{q} |\nabla v|^{2-4\alpha} \mathrm{d}x
  \leqslant & \Big(\int_{\Omega} u\Big)^{q} 
             \cdot \Big(\int_{\Omega}{\zeta}^{\frac{2}{1-{q}}} |\nabla v|^{\frac{2-4\alpha}{1-{q}}} \mathrm{d}x \Big)^{1-{q}} \notag \\
  \leqslant & \Big(\int_{\Omega} u\Big)^{q} 
             \cdot (\omega_n R^{n-1})^{1-q} \Big(\int_{\frac{R}{4}}^R | v_r|^{\frac{2-4\alpha}{1-{q}}} \mathrm{d}r \Big)^{1-{q}}.
\end{align}
Summing up (\ref{zeta2up})-(\ref{zeta2 up nablav 2-4alpha}), there exists a constant $C_7=C_7(q,\alpha,R,K_{1})>0$ such that

\begin{align}\label{1}
 \frac{1-{q}}{2}\int_{\Omega} \zeta^2 u^{{q}-2}\left|\nabla u\right|^2 \mathrm{d}x
    \leqslant &  \frac{1}{{q}} \frac{\mathrm{d}}{\mathrm{d} t} \int_{\Omega} \zeta^2 u^{q} \mathrm{d}x
               +(\frac{4}{(1-{q})}+1)  \int_{\Omega} u^{q} |\nabla \zeta|^2  \mathrm{d}x \notag \\
              &+ (2-q)K^2_{1} \int_{\Omega} {\zeta}^2 u^{q} |\nabla v|^{2-4\alpha} \mathrm{d}x \notag \\
    \leqslant &  \frac{1}{{q}} \frac{\mathrm{d}}{\mathrm{d}t} \int_{\Omega} \zeta^2 u^{q} \mathrm{d}x +C_7,
      \quad t \in (0,T_{\max})
\end{align} 
with the last inequality following from (\ref{vr semilinear estimate}) with $r_0=\frac{R}{4}$ and $\sigma=\frac{2-4\alpha}{1-{q}}$. We integrate both sides of (\ref{1}) and use (\ref{mass-1}) to find positive constants $C_8=C_8(R)$ and $C_9=C_9(q,\alpha,R,K_{1})$ such that
\begin{align}\label{uR-time}
& \int_t^{t+h} \int_{\frac{R}{2}}^R\left|\Big(u^{\frac{{q}}{2}}\Big)_r \right|^2 \mathrm{d}r \mathrm{d}s \notag \\
  \leqslant & C_8 \int_{t}^{t+h}\int_{\Omega} \zeta^2 u^{{q}-2}\left|\nabla u\right|^2 \mathrm{d}x \mathrm{d}s \notag \\
  \leqslant &  \frac{2C_8}{{q}(1-{q})} \int_{\Omega} \zeta^2 u^{q} \left(\cdot, t+h\right) \mathrm{d}x
             -\frac{2C_8}{{q}(1-{q})} \int_{\Omega} \zeta^2 u^{q} \left(\cdot, t\right) \mathrm{d}x  +C_7h \notag \\
  \leqslant &   \frac{2C_8|\Omega|^{1-{q}}}{{q}(1-{q})} \Big(\int_{\Omega} u\left(\cdot, t+h\right) \mathrm{d}x\Big)^{q} +C_7 \notag \\
  \leqslant & C_9.
\end{align}
Applying Gagliardo-Nirenberg inequality, and using (\ref{mass-1}) and (\ref{uR-time}) with $q=\frac{2}{3}$, we can find a positive constant $C_{10}=C_{10}(R)$ such that
\begin{align*}
\int_t^{t+h} u(R, s) \mathrm{d} s
  \leqslant & \int_t^{t+h} u^{\frac{5}{3}}(R, s) \mathrm{d} s  +h \notag \\
  \leqslant & \int_t^{t+h} \|u^{\frac{1}{3}}\|_{L^{\infty}\left((\frac{R}{2},R)\right)}^{5} \mathrm{d} s +1 \notag \\
  \leqslant & C_{10}\int_t^{t+h} \|(u^{\frac{1}{3}})_r\|_{L^2\left((\frac{R}{2},R)\right)}^2  \mathrm{d} s
              + C_{10}
\end{align*}
for all $t\in [0,T_{\max}-h)$ with $h=\min \left\{1, \frac{1}{2} T_{\max}\right\}$, which implies (\ref{uR}).
\end{proof}
The following lemma provides a spatio-temporal uniform bound for $v_r$ near the boundary.
\begin{lem}\label{vr-all}
Let $(u,v)$ be the classical solution of problem $(\ref{1.0.0})$ with $\tau=1$. Suppose that $f(\xi)$ satisfies $(\ref{f-lip})$ and $(\ref{f-global-1})$ with $\alpha \in (\frac{n-2}{2n},\frac{1}{2})$ and $K_{1}>0$. Then, for any radially symmetric $u_0$ and $v_0$ satisfy $(\ref{u_0})$ and $(\ref{v 0})$ respectively, there exists a constant $C_{11}=C_{11}(\alpha,R,K_{1})$ such that 
\begin{align}\label{vr}
\left|v_r(r, t)\right| \leqslant C_{11},
 \quad (r,t) \in \big(\frac{4R}{5},R\big) \times \left(0, T_{\max }\right).
\end{align}
\end{lem}
\begin{proof}
We choose $\zeta \in C^{\infty}([0, R])$ such that $0 \leqslant \zeta \leqslant 1$ fulfilling $\zeta \equiv 0$ in $\bar{B}_{\frac{R}{2}}(0)$ and $\zeta \equiv 1$ in $\bar{\Omega} \backslash B_\frac{3R}{4}(0)$. Multiplying the first equation in (\ref{1.0.0}) by $\zeta^2 u$  upon integration by parts, followed by the application of Young's inequality, we find that
\begin{align}\label{vr zeta2 up}
  &\frac{1}{2} \frac{\mathrm{d}}{\mathrm{d} t} \int_{\Omega} \zeta^2 u^2 \mathrm{d}x 
  + \int_{\Omega} \zeta^2 u^2 \mathrm{d}x \notag \\
  = & - \int_{\Omega} \zeta^2 \left|\nabla u\right|^2 \mathrm{d}x
      - \int_{\Omega} \zeta^2 uf(| \nabla v|^2) \nabla u \cdot \nabla v \mathrm{d}x \notag \\
  & -2 \int_{\Omega} \zeta u \nabla u \cdot \nabla \zeta \mathrm{d}x
    -2 \int_{\Omega} \zeta u^2 f(| \nabla v|^2) \nabla v \cdot \nabla \zeta \mathrm{d}x
    + \int_{\Omega} \zeta^2 u^2 \mathrm{d}x \notag \\
  \leqslant & \frac{3 K^2_{1}}{2} \int_{\Omega} \zeta^2 u^2 |\nabla v|^{2-4\alpha} \mathrm{d}x 
             +3 \int_{\Omega} u^2  |\nabla \zeta|^2 \mathrm{d}x 
             + \int_{\Omega} \zeta^2 u^2 \mathrm{d}x
\end{align}
for all $t \in (0,T_{\max})$. 

Next, we estimate the three terms on the right side of (\ref{vr zeta2 up}) to derive the boundedness of $\int_{\frac{3R}{4}}^R u^2 \mathrm{d}x$. For $q \in (0,1)$, by Young's inequality, and along with (\ref{vr semilinear estimate}) where $r_0=\frac{R}{2}$ and $\sigma=\frac{(2+{q})(2-4\alpha)}{{q}}>0$ because of $\frac{n-2}{2n}<\alpha<\frac{1}{2}$ and $0<q<1$, there exists a positive constant $C_{12}=C_{12}(p_1,\alpha,R)$ such that
\begin{align}\label{zeta2 up nablav2-4alpha}
\int_{\Omega} \zeta^2 u^2 |\nabla v|^{2-4\alpha} \mathrm{d}x
\leqslant & \int_{\Omega \backslash B_{\frac{R}{2}}(0)} u^{{2+{q}}} \mathrm{d}x
          +\int_{\Omega \backslash B_{\frac{R}{2}}(0)}\left|\nabla v\right|^{\frac{(2-4\alpha) (2+{q})}{{{q}}}} \mathrm{d}x \notag \\
\leqslant & \int_{\Omega \backslash B_{\frac{R}{2}}(0)} u^{{2+{q}}} \mathrm{d}x +C_{12}.
\end{align}
By means of Young's inequality, we have
\begin{align}
\int_{\Omega} u^2  |\nabla \zeta|^2 \mathrm{d}x
  \leqslant \int_{\Omega \backslash B_{\frac{R}{2}}(0)} u^{2+{q}} \mathrm{d}x
            + \int_{\Omega \backslash B_{\frac{R}{2}}(0)} \left|\nabla \zeta \right|^{\frac{2(2+{q})}{{{q}}}} \mathrm{d}x
\end{align}
and
\begin{align}\label{zeta2 up}
\int_{\Omega} \zeta^2 u^2 \mathrm{d}x
  \leqslant \int_{\Omega \backslash B_{\frac{R}{2}}(0)} u^{{2+{q}}} \mathrm{d}x
            +  \int_{\Omega \backslash B_{\frac{R}{2}}(0)} \zeta^{\frac{2(2+{q})}{{{q}}}} \mathrm{d}x.
\end{align}
The Gagliardo-Nirenberg inequality guarantees the existence of a positive constant $C_{13}=C_{13}(q,R)$ such that
\begin{align}\label{u}
\int_{\Omega \backslash B_{\frac{R}{2}}(0)} u^{{2+{q}}} \mathrm{d}x
  \leqslant &  \omega_n R^{n-1}\int_{\frac{R}{2}}^{R} u^{{2+{q}}}(r,t) \mathrm{d}r
  =\omega_n R^{n-1} \|u^{\frac{q}{2}}\|_{L^{\frac{2(2+p_1)}{q}}\left((\frac{R}{2},R)\right)}^{\frac{2(2+p_1)}{q}} \notag \\
  \leqslant & C_{13}\|(u^{\frac{q}{2}})_r\|_{L^2\left((\frac{R}{2},R)\right)}^2  \|u^{\frac{q}{2}}\|_{L^{\frac{2}{q}}\left((\frac{R}{2},R)\right)}^{\frac{4}{q}} 
             + C_{13}\|u^{\frac{q}{2}}\|_{L^{\frac{2}{q}}\left((\frac{R}{2},R)\right)}^{\frac{2(2+p_1)}{q}}.
\end{align}
Besides, we have
\begin{align*}
\|u^{\frac{q}{2}}\|_{L^{\frac{2}{q}}(\frac{R}{2},R)}^{\frac{2}{q}} 
  = \int_{\frac{R}{2}}^{R} u(r,t) \mathrm{d}r
  \leqslant \Big(\frac{2}{R}\Big)^{n-1} \int_{\frac{R}{2}}^{R} r^{n-1}u(r,t) \mathrm{d}r 
  = \Big(\frac{2}{R}\Big)^{n-1}\frac{1}{\omega_n} \int_{\Omega} u_0 \mathrm{d}x.
\end{align*}
 Combining this with (\ref{uR-time}) and (\ref{u}), there exists a constant $C_{14}=C_{14}(p_1,R)>0$ such that
\begin{align}\label{u2+p1}
\int_t^{t+h}\int_{\Omega \backslash B_{\frac{R}{2}}(0)} u^{{2+{q}}} \mathrm{d}x \mathrm{d}s \leqslant C_{14},
\quad t \in [0,T_{\max}-h).
\end{align}
Substituting (\ref{zeta2 up nablav2-4alpha})-(\ref{zeta2 up}) into (\ref{vr zeta2 up}) implies the existence of a positive constant $C_{15}=C_{15}(q,\alpha,R)$ such that
\begin{align}\label{zeta2 up-ineq}
&\frac{1}{2} \frac{\mathrm{d}}{\mathrm{d} t} \int_{\Omega} \zeta^2 u^2 \mathrm{d}x 
+ \int_{\Omega} \zeta^2 u^2 \mathrm{d}x \notag \\
 \leqslant & \left(\frac{3K^2_{1}}{2}+3\right)\int_{\Omega \backslash B_{\frac{R}{2}}(0)} u^{{2+{q}}} \mathrm{d}x +C_{15} \notag \\
 \overset{\Delta}{=} & g(t) ,
 \quad t \in (0,T_{\max}).
\end{align}
Combining (\ref{u2+p1}) with (\ref{zeta2 up-ineq}), and applying Lemma~\ref{lem-3.1.0.2} to (\ref{zeta2 up-ineq}), we have
\begin{align}\label{u2-R}
\int_{\frac{3R}{4}}^R u^2 \mathrm{d}x 
  \leqslant \Big(\frac{4}{3R}\Big)^{n-1}\frac{1}{\omega_n}\int_{\Omega}u_0^2\mathrm{d}x + \Big(\frac{4}{3R}\Big)^{n-1}\frac{1}{\omega_n}\frac{((3K^2_{1}+6)C_{14}+2C_{15}}{1-e^{-2}}
\end{align}
for all $t \in (0,T_{\max})$. 

We fix $\chi \in C^{\infty}([0, R])$ such that $0 \leqslant \chi \leqslant 1$ fulfilling $\chi \equiv 0$ in $\left[0, \frac{3R}{4}\right]$ and $\chi \equiv 1$ in $[\frac{4R}{5}, R]$. According to (\ref{b}), (\ref{u2-R}) and (\ref{vr semilinear estimate}) with $r_0=\frac{3R}{4}$ and $\sigma=2$, there exists a constant $C_{16}=C_{16}(\alpha,R,K_{1})>0$ such that
\begin{align}\label{b-2}
\left\|b(\cdot, t)\right\|_{L^2 \left((\frac{3R}{4}, R)\right)} \leqslant C_{16},
\quad t \in (0,T_{\max}).
\end{align}
Using (\ref{b-2}), (\ref{Delta-u_0}) and (\ref{Delta-u}), we derive that
\begin{align*}
  \left\|v_r(\cdot, t)\right\|_{L^{\infty}\left((\frac{4R}{5}, R)\right)} 
  \leqslant &\left\|\partial_r\big(\chi \cdot\left(v(\cdot, t)-M\right)\big)\right\|_{L^{\infty}\left((\frac{3R}{4}, R)\right)} \notag \\
  \leqslant & C_4\left\|\chi \cdot\left(v_0-M\right)\right\|_{W^{1, \infty}\left((\frac{3R}{4}, R)\right)}+C_5C_{16} \int_0^{\infty}\left(1+t^{-\frac{3}{4}}\right) e^{-\lambda t} \mathrm{d} t
\end{align*}
for all $t \in (0,T_{\max})$ which implies (\ref{vr}).
\end{proof}
\begin{lem}\label{p-q}
Let $\Omega=B_R(0) \subset \mathbb{R}^n$ $(n \geqslant 3)$ and $(u,v)$ be the classical solution of problem $(\ref{1.0.0})$ with $\tau=1$. Suppose that $f(\xi)$ satisfies $(\ref{f-lip})$ and $(\ref{f-global-1})$ with $\alpha>\frac{n-2}{2n}$ and $K_1>0$. Assume that $(\ref{u_0})$ and $(\ref{v 0})$ are valid. Then for $q > \max \left\{1, \frac{n-2}{2}\right\}$ and $p \in \left[2 q+2, \frac{(2 n-2) q}{n-2}+1\right]$, there exists a constant $C_{17}>0$ such that 
\begin{align}\label{p-n}
\|\nabla v\|_{L^p(\Omega)} 
  \leqslant C_{17}\left\||\nabla v|^{q-1}  |D^2v| \right\|_{L^2(\Omega)}^{\frac{p-2}{p q}}
            +C_{17}.
\end{align}
If $n=2$, for $q \geqslant 1$ and $p \geqslant 2 q+2$, we also have
\begin{align}\label{p-2}
\|\nabla v\|_{L^p(\Omega)} 
  \leqslant C_{17}\left\||\nabla v|^{q-1}  |D^2v| \right\|_{L^2(\Omega)}^{\frac{p-2}{p q}}
            +C_{17}.
\end{align}
\end{lem}
\begin{proof}
For all $\frac{n-2}{2n}<\alpha<\frac{1}{2}$, by means of (\ref{vr}) and $p>1$, we obtain
\begin{align}\label{partialv}
\int_{\partial \Omega} |\nabla v|^{p-2}v\frac{\partial v}{\partial n} \mathrm{d}x
  \leqslant |\partial \Omega| \|v_0\|_{L^{\infty}(\Omega)} \left\|v_ r\right\|_{L^{\infty}\left((\frac{4R}{5}, R) \times\left(0, T_{\max }\right)\right)}^{p-1}
  \leqslant |\partial \Omega|  \|v_0\|_{L^{\infty}(\Omega)} {C_{11}}^{p-1}
\end{align}
for all $t \in (0,T_{\max})$. Similar to Lemma 3.3 in \cite{2015-ZAMP-WangXiang}, togther with (\ref{partialv}) and (\ref{vmax}), this leads to (\ref{p-n}). Combining (\ref{partialv}) with (\ref{nablavs}) and (\ref{vmax}), and following a proof similar to that of Lemma 2.1 in \cite{2015-MMMAS-LiSuenWinklerXue}, we can deduce (\ref{p-2}).
\end{proof}
To derive a uniform estimate for $u$, we construct a following differential inequality.
\begin{lem}\label{lem-3.1.2}
Let $(u,v)$ be the classical solution of problem $(\ref{1.0.0})$ with $\tau=1$. Suppose that $f(\xi)$ satisfies $(\ref{f-lip})$ and $(\ref{f-global-1})$ with $\alpha \in (\frac{n-2}{2n},\frac{1}{2})$ and $K_{1}>0$. Assume that $u_0$ and $v_0$ are radially symmetric and satisfy $(\ref{u_0})$ and $(\ref{v 0})$ respectively. Then, for any 
\begin{align}\label{k}
k>\max\{1,\frac{n-2}{2},\frac{n}{2}-\alpha n-2\alpha\}
\end{align}
and 
\begin{align}\label{p-1}
\max\{1,k+1-\frac{2}{n}\}<p<\frac{2k+4\alpha}{n(1-2\alpha)},
\end{align}
there exist constants $C_{18}=C_{18}(\alpha,R,K_{1},k,p)>0$ and $C_{19}=C_{19}(\alpha,R,K_{1},k,p)>0$ such that 
\begin{align}\label{all-2}
 & \frac{\mathrm{d}}{\mathrm{d} t} \left(\frac{1}{p}\int_{\Omega}u^p \mathrm{d}x  +  \frac{1}{2k}\int_{\Omega} |\nabla v|^{2k}\mathrm{d}x  \right)
   + C_{18} \int_{\Omega} |\nabla v|^{2(k-1)} |D^2v|^2 \mathrm{d}x  +  C_{18}\int_{\Omega} u^{p-2} |\nabla u|^2 \mathrm{d}x  \notag \\
 \leqslant & \frac{1}{2}\int_{\partial \Omega} |\nabla v|^{2(k-1)} \frac{\partial |\nabla v|^2}{\partial n} \mathrm{d}x 
             + \int_{\partial \Omega} |\nabla v|^{2(k-1)} uv \big|\frac{\partial v}{\partial n}\big| \mathrm{d}x +C_{19}
\end{align}
for all $t \in (0,T_{\max})$.
\end{lem}
\begin{proof}
We first verify (\ref{p-1}) using (\ref{k}), which justifies the choice of $p$. Because of $k>\frac{n}{2}-\alpha n-2\alpha$, we have
\begin{align}\label{k-p-1}
\frac{2k+4\alpha}{n(1-2\alpha)}
>\frac{2(\frac{n}{2}-\alpha n-2\alpha)+4\alpha}{n(1-2\alpha)}
=1.
\end{align}
Due to $\frac{n-2}{2n}<\alpha<\frac{1}{2}$ and $k>1$, we infer that
\begin{align}\label{k-p-2}
\Big(k+1-\frac{2}{n}\Big)- \frac{2k+4\alpha}{n(1-2\alpha)}
=\frac{2n(\frac{n-2}{2n}-\alpha)(k+1)}{n(1-2\alpha)}
<0.
\end{align}
Combining (\ref{k-p-1}) with (\ref{k-p-2}), (\ref{p-1}) is justified. 

By $p>k+1-\frac{2}{n}>(1-\frac{2}{n})k+1-\frac{2}{n}$, we deduce that
\begin{align}\label{theta1}
\Big(\frac{p}{2(k+1)}-\frac{p}{2}\Big)-\Big(\frac{1}{2}-\frac{1}{n}-\frac{p}{2}\Big)
>\frac{(1-\frac{2}{n})(k+1)}{2(k+1)}-\Big(\frac{1}{2}-\frac{1}{n}\Big)
=0.
\end{align}
The assumption $k>\frac{n}{2}-\alpha n-2\alpha>\frac{n}{2}-\alpha n-1$ ensures that
\begin{align}\label{theta2}
&\Big(\frac{k+2\alpha}{2(k+1)}-\frac{p}{2}\Big)-\Big(\frac{1}{2}-\frac{1}{n}-\frac{p}{2}\Big) \notag\\
=&\frac{k+2\alpha}{2(k+1)}-\Big(\frac{1}{2}-\frac{1}{n}\Big) \notag\\
=&\frac{\frac{2}{n}(k+1)-(1-2\alpha)}{2(k+1)} \notag\\
=&\frac{1}{n}-\frac{1-2\alpha}{2(k+1)} \notag\\
>&\frac{1}{n}-\frac{1-2\alpha}{2(\frac{n}{2}-\alpha n)} \notag\\
=&0.
\end{align}
Moreover, 
\begin{align}\label{<0}
\frac{1}{2}-\frac{1}{n}-\frac{p}{2}<-\frac{1}{n}<0
\end{align}
and
\begin{align}\label{<0-1}
\frac{k+2\alpha}{2(k+1)}-\frac{p}{2}<\frac{k+2\alpha}{2(k+1)}-\frac{1}{2}\Big(1-\frac{1-2\alpha}{k+1}\Big)=0
\end{align}
as consequences of $p>1>1-\frac{1-2\alpha}{k+1}$. 
According to $p<\frac{2k+4\alpha}{n(1-2\alpha)}$ and (\ref{<0}), we have
\begin{align}\label{<1-1}
\frac{\frac{k+2\alpha}{2(k+1)}-\frac{p}{2}}{\frac{1}{2}-\frac{1}{n}-\frac{p}{2}} 
+ \frac{1-2\alpha}{k+1}-1
=\frac{\frac{1}{n}-\frac{1-2\alpha}{k+1}(\frac{1}{n}+\frac{p}{2})}{\frac{1}{2}-\frac{1}{n}-\frac{p}{2}}
<\frac{\frac{1}{n}-\frac{1-2\alpha}{k+1}(\frac{1}{n}+\frac{k+2\alpha}{n(1-2\alpha)})}{\frac{1}{2}-\frac{1}{n}-\frac{p}{2}}
=0.
\end{align}
Combining $p>k+1-\frac{2}{n}$ with (\ref{<0}), we deduce that
\begin{align}\label{<1-2}
\frac{2}{p} \cdot \frac{\frac{p}{2(k+1)}-\frac{p}{2}}{\frac{1}{2}-\frac{1}{n}-\frac{p}{2}}+\frac{k-1}{k+1}-1
= &\frac{-\frac{k}{k+1}+\frac{k-1}{k+1}(\frac{1}{2}-\frac{1}{n}-\frac{p}{2})}{\frac{1}{2}-\frac{1}{n}-\frac{p}{2}}-1 \notag \\
<&\frac{-\frac{k}{k+1}+\frac{k-1}{k+1}(\frac{1}{2}-\frac{1}{n}-(\frac{k}{2}+\frac{1}{2}-\frac{1}{n}))}{\frac{1}{2}-\frac{1}{n}-\frac{p}{2}}-1 \notag \\
= &\frac{-\frac{k}{2}}{\frac{1}{2}-\frac{1}{n}-\frac{p}{2}}-1 \notag \\
= &\frac{\frac{p}{2}-\frac{1}{2}+\frac{1}{n}-\frac{k}{2}}{\frac{1}{2}-\frac{1}{n}-\frac{p}{2}} \notag \\
<&0.
\end{align}

Multiplying the first equation in (\ref{1.0.0}) by $u^{p-1}$, integrating by parts and applying (\ref{f-global-1}), we obtain that
\begin{align}\label{u2}
  \frac{1}{p} \frac{\mathrm{d}}{\mathrm{d}t} \int_{\Omega}u^p \mathrm{d}x 
   = & -(p-1)\int_{\Omega} u^{p-2} |\nabla u|^2 \mathrm{d}x - (p-1)\int_{\Omega} u^{p-1} f(|\nabla v|^2) \nabla u \cdot \nabla v \mathrm{d}x \notag \\ 
   \leqslant & -\frac{p-1}{2} \int_{\Omega} u^{p-2} |\nabla u|^2 \mathrm{d}x 
              + \frac{K^2_{1}(p-1)}{2} \int_{\Omega} \frac{u^p}{(1+|\nabla v|^2)^{2\alpha}} |\nabla v|^2 \mathrm{d}x  \notag \\
   \leqslant & -\frac{p-1}{2} \int_{\Omega} u^{p-2} |\nabla u|^p \mathrm{d}x 
                + \frac{K^2_{1}(p-1)}{2} \int_{\Omega} {u^p} |\nabla v|^{2-4\alpha} \mathrm{d}x
\end{align}
for all $t \in (0,T_{\max})$. Using the second equation in (\ref{1.0.0}), we infer that
\begin{align}\label{v2-4alpha}
\frac{1}{2k} \frac{\mathrm{d}}{\mathrm{d}t} \int_{\Omega} |\nabla v|^{2k} \mathrm{d}x 
   = &\int_{\Omega}  |\nabla v|^{2(k-1)} \nabla v \cdot \nabla v_t \mathrm{d}x  \notag \\
   = &\int_{\Omega}  |\nabla v|^{2(k-1)} \nabla v \cdot \nabla (\Delta v-uv) \mathrm{d}x  \notag \\
   = &\int_{\Omega}  |\nabla v|^{2(k-1)} \nabla v \cdot \nabla \Delta v \mathrm{d}x 
      - \int_{\Omega}  |\nabla v|^{2(k-1)} \nabla v \cdot \nabla (uv) \mathrm{d}x
\end{align}
for all $t\in (0,T_{\max})$. Due to $\nabla v \cdot \nabla \Delta v= \frac{1}{2} \Delta |\nabla v|^2-|D^2v|^2$, we have
\begin{align}\label{v2-4alpha-1}
  &\int_{\Omega}  |\nabla v|^{2(k-1)} \nabla v \cdot \nabla \Delta v \mathrm{d}x 
     = \frac{1}{2} \int_{\Omega}  |\nabla v|^{2(k-1)} \Delta |\nabla v|^2 \mathrm{d}x 
       - \int_{\Omega}  |\nabla v|^{2(k-1)} |D^2v|^2 \mathrm{d}x  \notag \\
     = & -\frac{k-1}{2} \int_{\Omega} |\nabla v|^{2(k-2)} |\nabla |\nabla v|^2|^2 \mathrm{d}x 
      + \frac{1}{2} \int_{\partial \Omega} |\nabla v|^{2(k-1)} \frac{\partial |\nabla v|^2}{\partial n} \mathrm{d}x \notag \\
     & -\int_{\Omega} |\nabla v|^{2(k-1)} |D^2 v|^2 \mathrm{d}x .
\end{align}
Integrating by parts, we get
\begin{align}\label{v2-4alpha-2}
&\int_{\Omega} |\nabla v|^{2(k-1)} \nabla v \cdot \nabla (uv) \mathrm{d}x \notag \\
  =& -\int_{\Omega} |\nabla v|^{2(k-1)} uv \Delta v \mathrm{d}x  
    -(k-1) \int_{\Omega} |\nabla v|^{2(k-2)} uv \nabla |\nabla v|^2 \cdot \nabla v \mathrm{d}x  \notag \\
   & + \int_{\partial \Omega} |\nabla v|^{2(k-1)} uv \frac{\partial v}{\partial n} \mathrm{d}x.
\end{align}
By means of Hölder's inequality, we find that
\begin{align}\label{v2-4alpha-2-1}
\int_{\Omega} |\nabla v|^{2(k-1)} uv \Delta v \mathrm{d}x
  \leqslant & \frac{1}{2n} \int_{\Omega} |\nabla v|^{2(k-1)} |\Delta v|^2 \mathrm{d}x 
            + \frac{n}{2}\int_{\Omega} |\nabla v|^{2(k-1)} u^2 v^2 \mathrm{d}x \notag \\
  \leqslant & \frac{1}{2} \int_{\Omega} |\nabla v|^{2(k-1)} |D^2 v|^2 \mathrm{d}x 
            + \frac{n}{2}\int_{\Omega} |\nabla v|^{2(k-1)} u^2 v^2 \mathrm{d}x 
\end{align}
and
\begin{align}\label{v2-4alpha-2-2}
& (k-1)\int_{\Omega} |\nabla v|^{2(k-2)} uv \nabla |\nabla v|^2 \cdot \nabla v \mathrm{d}x \notag \\
\leqslant & \frac{k-1}{4}\int_{\Omega} |\nabla v|^{2(k-2)} |\nabla |\nabla v|^2|^2 \mathrm{d}x
            +(k-1)  \int_{\Omega} |\nabla v|^{2(k-1)} u^2 v^2 \mathrm{d}x .
\end{align}
Inserting (\ref{v2-4alpha-1})-(\ref{v2-4alpha-2-2}) into (\ref{v2-4alpha}), we obtain
\begin{align*}
&\frac{1}{2k} \frac{\mathrm{d}}{\mathrm{d}t}\int_{\Omega} |\nabla v|^{2k} \mathrm{d}x 
+\frac{k-1}{4} \int_{\Omega} |\nabla v|^{2(k-2)} |\nabla |\nabla v|^2|^2 \mathrm{d}x
+\frac{1}{2}\int_{\Omega} |\nabla v|^{2(k-1)} |D^2 v|^2 \mathrm{d}x  \notag \\
  \leqslant & (k-1+\frac{n}{2}) \int_{\Omega} |\nabla v|^{2(k-1)} u^2 v^2 \mathrm{d}x
             +\frac{1}{2} \int_{\partial \Omega} |\nabla v|^{2(k-1)} \frac{\partial |\nabla v|^2}{\partial n} \mathrm{d}x 
             - \int_{\partial \Omega} |\nabla v|^{2(k-1)} uv \frac{\partial v}{\partial n} \mathrm{d}x
\end{align*}
for all $t \in (0,T_{\max})$. Combining this with (\ref{u2}), and using $|\nabla|\nabla v|^2|^2=4|\nabla v|^2|D^2 v|^2$ and (\ref{vmax}), we deduce that
\begin{align}\label{all-1}
&\frac{\mathrm{d}}{\mathrm{d}t}\left(\frac{1}{2k}\int_{\Omega} |\nabla v|^{2k} \mathrm{d}x 
   + \frac{1}{p} \int_{\Omega}u^p \mathrm{d}x \right)
   +\frac{p-1}{2}\int_{\Omega} u^{p-2} |\nabla u|^2 \mathrm{d}x 
   +(k-\frac{1}{2})\int_{\Omega} |\nabla v|^{2(k-1)} |D^2 v|^2 \mathrm{d}x  \notag \\
\leqslant & \Big(k-1+\frac{n}{2}\Big)\|v_0\|_{L^{\infty}(\Omega)}^2 \int_{\Omega} u^2 |\nabla v|^{2(k-1)} \mathrm{d}x
             +\frac{K^2_{1}(p-1)}{2} \int_{\Omega}  u^p |\nabla v|^{2-4\alpha} \mathrm{d}x \notag \\
           &+\frac{1}{2} \int_{\partial \Omega} |\nabla v|^{2(k-1)} \frac{\partial |\nabla v|^2}{\partial n} \mathrm{d}x 
             + \int_{\partial \Omega} |\nabla v|^{2(k-1)} uv \big|\frac{\partial v}{\partial n}\big| \mathrm{d}x 
\end{align}
for all $t \in (0,T_{\max})$. 

In the following, we estimate the first term on the right side of (\ref{all-1}). Applying Hölder's inequality, we have
\begin{align}\label{u2 nabla v2-4alpha}
  \int_{\Omega} u^2 |\nabla v|^{2(k-1)}  \mathrm{d}x
    \leqslant \left(\int_{\Omega}|\nabla v|^{2(k+1)}\mathrm{d}x\right)^{\frac{k-1}{k+1}}
              \left(\int_{\Omega}u^{k+1}\mathrm{d}x\right)^{\frac{2}{k+1}}.
\end{align}
According to Lemma~\ref{p-q} with $q=k$ and $p=2k+2$, we have  
\begin{align}\label{v-lemma p-q}
 \left(\int_{\Omega}|\nabla v|^{2(k+1)}\mathrm{d}x\right)^{\frac{k-1}{k+1}} = \big\|\nabla v\big\|_{L^{2k+2}(\Omega)}^{2k-2}
  \leqslant C_{17} \left\||\nabla v|^{k-1} | D^2v| \right\|_{L^2(\Omega)}^{\frac{2(k-1)}{k+1}}+C_{17}.
\end{align}
The Gagliardo-Nirenberg inequality implies the existence of a positive constant $C_{20}=C_{20}(p,k)$ such that
\begin{align}\label{u3-2alpha-G-N}
  \left(\int_{\Omega}u^{k+1}\mathrm{d}x\right)^{\frac{2}{k+1}}
  = &\|u^{\frac{p}{2}}\|_{L^{\frac{2(k+1)}{p}}(\Omega)}^{\frac{4}{p}}  \notag \\
  \leqslant & C_{20}\|\nabla u^{\frac{p}{2}}\|_{L^2(\Omega)}^{\frac{4\theta_1}{p}} 
              \|u^{\frac{p}{2}}\|_{L^{\frac{2}{p}}(\Omega)}^{\frac{4(1-\theta_1)}{p}}
              +C_{20}\|u^{\frac{p}{2}}\|_{L^{\frac{2}{p}}(\Omega)}^{\frac{4}{p}},
\end{align}
where $\theta_1 = \frac{\frac{p}{2(k+1)}-\frac{p}{2}}{\frac{1}{2}-\frac{1}{n}-\frac{p}{2}} \in (0,1)$ because of (\ref{theta1}) and (\ref{<0}). Inserting (\ref{v-lemma p-q}) and (\ref{u3-2alpha-G-N}) into (\ref{u2 nabla v2-4alpha}), and using (\ref{a-b}) due to (\ref{<1-2}), for any $\varepsilon_1>0$, there exists a positive constant $C_{21}=C_{21}(\alpha,R,K_{1},\varepsilon_1,k,p)$ such that
\begin{align}\label{var1}
&(k-1+\frac{n}{2})\|v_0\|_{L^{\infty}(\Omega)}^2  \int_{\Omega} u^2 |\nabla v|^{2(k-1)}  \mathrm{d}x \notag \\
\leqslant & \varepsilon_1 \int_{\Omega} |\nabla v|^{2(k-1)} |D^2 v|^2 \mathrm{d}x
           + \varepsilon_1 \int_{\Omega} u^{p-2} |\nabla u|^2 \mathrm{d}x
           + C_{21}.
\end{align}

For the second term on the right side of (\ref{all-1}), by means of Hölder's inequality, we get
\begin{align}\label{up nabla2-4alpha}
\int_{\Omega}  u^p |\nabla v|^{2-4\alpha} \mathrm{d}x
\leqslant \left(\int_{\Omega} |\nabla v|^{2(k+1)} \mathrm{d}x \right)^{\frac{1-2\alpha}{k+1}}
          \left(\int_{\Omega} u^{\frac{p(k+1)}{k+2\alpha}} \right)^{\frac{k+2\alpha}{k+1}}.
\end{align}
According to Lemma~\ref{p-q} with $q=k$ and $p=2k+2$, we have
\begin{align}\label{nablav2k+2}
\left(\int_{\Omega} |\nabla v|^{2(k+1)} \mathrm{d}x \right)^{\frac{1-2\alpha}{k+1}}
=\|\nabla v\|_{L^{2k+2}(\Omega)}^{2(1-2\alpha)}
\leqslant C_{17}\left\||\nabla v|^{q-1}  |D^2v| \right\|_{L^2(\Omega)}^{\frac{2(1-2\alpha)}{k+1}} +C_{17}.
\end{align}
Applying Gagliardo-Nirenberg inequality, there exists a constant $C_{22}=C_{22}(k,\alpha,p)>0$ such that
\begin{align}\label{u-G-N-1}
\left(\int_{\Omega} u^{\frac{p(k+1)}{k+2\alpha}} \right)^{\frac{k+2\alpha}{k+1}}
=\|u^{\frac{p}{2}}\|_{L^{\frac{2(k+1)}{k+2\alpha}}(\Omega)}^2
\leqslant C_{22} \|\nabla u^{\frac{p}{2}}\|_{L^2(\Omega)}^{2\theta_2}  \|u^{\frac{p}{2}}\|_{L^{\frac{2}{p}}(\Omega)}^{2(1-\theta_2)}
          +C_{22} \|u^{\frac{p}{2}}\|_{L^{\frac{2}{p}}(\Omega)}^{2},
\end{align}
where $\theta_2=\frac{\frac{k+2\alpha}{2(k+1)}-\frac{p}{2}}{\frac{1}{2}-\frac{1}{n}-\frac{p}{2}} \in (0,1)$ because of (\ref{theta2}), (\ref{<0}) and (\ref{<0-1}). 
Inserting (\ref{u-G-N-1}) and (\ref{nablav2k+2}) into (\ref{up nabla2-4alpha}), and using (\ref{a-b}) due to (\ref{<1-1}), for any $\varepsilon_2>0$, there exists a constant $C_{23}=C_{23}(\alpha,R,K_{1},\varepsilon_2,k,p )>0$ such that
\begin{align}\label{var2}
\frac{K^2_{1}(p-1)}{2} \int_{\Omega}  u^p |\nabla v|^{2-4\alpha} \mathrm{d}x
\leqslant \varepsilon_2 \int_{\Omega} |\nabla v|^{2(k-1)} |D^2 v|^2 \mathrm{d}x
           + \varepsilon_2 \int_{\Omega} u^{p-2} |\nabla u|^2 \mathrm{d}x
           + C_{23}.
\end{align}
Inserting (\ref{var1}) and (\ref{var2}) into (\ref{all-1}), and choosing $\varepsilon_1$ and $\varepsilon_2$ small enough yields (\ref{all-2}).
\end{proof}
Considering (\ref{all-2}), the associated boundary integrals will be estimated by making use of the radial symmetry of $v$. Thus, the following $L^p$ estimate for $u$ can be derived.
\begin{lem}\label{lem-3.2.3}
Let $(u,v)$ be the classical solution of problem $(\ref{1.0.0})$ with $\tau=1$. Suppose that $f(\xi)$ satisfies $(\ref{f-lip})$ and $(\ref{f-global-1})$ with $\alpha \in (\frac{n-2}{2n},\frac{1}{2})$ and $K_{1}>0$. Assume that $u_0$ and $v_0$ are radially symmetric and satisfy $(\ref{u_0})$ and $(\ref{v 0})$ respectively. Then, for any $k$ and $p$ satisfy (\ref{k}) and (\ref{p-1}) respectively, there exists a positive constant $C_{24}=C_{24}(\alpha,R,K_{1},p,k)>0$ such that 
\begin{align}\label{u-1}
\int_{\Omega}u^p \mathrm{d}x \leqslant C_{24},
\quad t \in (0,T_{\max}).
\end{align}
\end{lem}
\begin{proof} 

Applying Gagliardo-Nirenberg inequality and Young's inequality, using (\ref{nablavs}), there exists a constant $C_{26}=C_{26}(k)>0$ such that
\begin{align}\label{nablav4-4alpha} 
\int_{\Omega} |\nabla v|^{2k}\mathrm{d}x 
   = & \left\||\nabla v|^{k}\right\|_{L^2(\Omega)}^2 \notag \\
   \leqslant & C_{26} \left\|\nabla |\nabla v|^{k}\right\|_{L^2(\Omega)}^{2\theta_3}
               \left\| |\nabla v|^{k}\right\|_{L^{\frac{a}{k}}(\Omega)}^{2(1-\theta_3)}
               + C_{26} \left\| |\nabla v|^{k}\right\|_{L^{\frac{a}{k}}(\Omega)}^2 \notag \\
   \leqslant & C_{26}C_2^{2k(1-\theta_3)} \left\|\nabla |\nabla v|^{k}\right\|_{L^2(\Omega)}^{2\theta_3} +C_{26}C_2^{2k} \notag \\
   \leqslant & k^{2\theta_3}C_{26}C_2^{2k(1-\theta_3)} \left\||\nabla v|^{k-1} |D^2v| \right\|_{L^2(\Omega)}^{2\theta_3} +C_{26}C_2^{2k}  ,
\end{align}
where $\theta_3=\frac{2kn-na}{2a-na+2kn} \in (0,1)$ because of $k>1$ and $a>1$. Similarly, by Gagliardo-Nirenberg inequality, we can find a constant $C_{27}=C_{27}(p)>0$ such that
\begin{align}\label{u2-G-N}
\int_{\Omega}u^p \mathrm{d}x 
  \leqslant & C_{27}\| \nabla u^{\frac{p}{2}}\|_{L^2(\Omega)}^{2\theta_4}   \|u^{\frac{p}{2}}\|_{L^{\frac{2}{p}}(\Omega)}^{2(1-\theta_4)}  +   C_{27}\|u^{\frac{p}{2}}\|_{L^{\frac{2}{p}}(\Omega)}^2 ,
\end{align}
where $\theta_4=\frac{\frac{1}{2}-\frac{p}{2}}{\frac{1}{2}-\frac{1}{n}-\frac{p}{2}} \in (0,1)$ due to (\ref{<0}) and $p>1$. Combining (\ref{all-2}) with (\ref{nablav4-4alpha}) and (\ref{u2-G-N}), using Young's inequality, there exists a constant $C_{28}=C_{28}(\alpha,k,p,K_{1},R)>0$ such that
\begin{align}\label{all-3}
 & \frac{\mathrm{d}}{\mathrm{d} t} \left(\frac{1}{p}\int_{\Omega}u^p \mathrm{d}x  +  \frac{1}{2k}\int_{\Omega} |\nabla v|^{2k}\mathrm{d}x  \right)
   + \frac{1}{p}\int_{\Omega}u^p \mathrm{d}x  +  \frac{1}{2k}\int_{\Omega} |\nabla v|^{2k}\mathrm{d}x  \notag \\
 \leqslant & \frac{1}{2}\int_{\partial \Omega} |\nabla v|^{2(k-1)} \frac{\partial |\nabla v|^2}{\partial n} \mathrm{d}x 
             + \int_{\partial \Omega} |\nabla v|^{2(k-1)} uv |\frac{\partial v}{\partial n}| \mathrm{d}x +C_{28}.
\end{align}

For the boundary integrals, using the second equation in (\ref{1.0.0}) and $v=M$ on $\partial \Omega$, we deduce that
\begin{align*}
  \frac{\partial\left|\nabla v\right|^2}{\partial \nu}
= &2 v_r v_{ r r} \notag \\
= &2 v_{ r} \cdot\left\{v_{ r r}+\frac{1}{R} v_{r}\right\}-\frac{2}{R} v_{r}^2 \notag \\
= &2 u v v_{ r}-\frac{2}{R } v_{ r}^2 \notag \\
\leqslant & 2 u v v_{ r},
 \quad t \in\left(0, T_ {\max }\right)   .
\end{align*}
Combining this with (\ref{uR}), (\ref{vr}) and $k>1$, we obtain
\begin{align}\label{partial nablav}
\int_t^{t+h}\int_{\partial \Omega} |\nabla v|^{2(k-1)} \frac{\partial |\nabla v|^2}{\partial n} \mathrm{d}x \mathrm{d}s
  \leqslant & 2M\int_t^{t+h} u(R,s) v_r^{2k-1}(R,s) \mathrm{d}x \mathrm{d}s \notag \\
  \leqslant & 2M\| v_r\|_{L^{\infty}\left((\frac{4R}{5},R)\right)}^{2k-1} \int_t^{t+h} u(R,s) \mathrm{d}x \mathrm{d}s \notag \\
  \leqslant & 2MC_{11}^{2k-1} C_6,
  \quad t \in [0,T_{\max}-h). 
\end{align}
Similarly, there exists a constant $C_{25}=C_{25}(\alpha,R,K_{1},k)>0$ such that
\begin{align}\label{partial uv}
\int_t^{t+h}\int_{\partial \Omega} |\nabla v|^{2(k-1)} uv |\frac{\partial v}{\partial n}| \mathrm{d}x \mathrm{d}s \leqslant C_{25},
\quad t \in [0,T_{\max}-h).
\end{align}
We define 
\begin{align*}
y(t)=\frac{1}{p} \int_{\Omega}(1+u)^p \mathrm{d}x 
    +\frac{1}{2k}\int_{\Omega} |\nabla v|^{2k} \mathrm{d}x  
\end{align*}
and 
\begin{align*}
g(t):=  \frac{1}{2} \int_{\partial \Omega} |\nabla v|^{2(k-1)} \frac{\partial |\nabla v|^2}{\partial n} \mathrm{d}x 
        + \int_{\partial \Omega} |\nabla v|^{2(k-1)} uv \left|\frac{\partial v}{\partial n}\right| \mathrm{d}x  +C_{28},
\end{align*}
which satisfies $\int_t^{t+h}  g(s) \mathrm{d}s \leqslant C_{29}(\alpha,R,K_{1},k,p) $ for all $t \in (0,T_{\max}-h)$ with $h=\min \left\{1, \frac{1}{2} T_{\max }\right\}$, as given by (\ref{partial nablav}) and (\ref{partial uv}). It follows from (\ref{all-3}) that
\begin{align*}
y^{\prime}(t)+ y(t) \leqslant g(t), \quad t \in (0,T_{\max}),
\end{align*}
which implies (\ref{u-1}) by means of Lemma~\ref{lem-3.1.0.2}.
\end{proof}
\begin{lem}\label{lem-3.2.4}
Let $(u,v)$ be the classical solution of problem $(\ref{1.0.0})$ with $\tau=1$. Suppose that $f(\xi)$ satisfies $(\ref{f-lip})$ and $(\ref{f-global-1})$ with $\alpha \geqslant \frac{1}{2}$ and $K_{1}>0$. Assume that $u_0$ and $v_0$ are radially symmetric and satisfy $(\ref{u_0})$ and $(\ref{v 0})$ respectively. Then, for any $p>2$, there exists a positive constant $C_{30}=C_{30}(\alpha,R,K_{1},p)>0$ such that 
\begin{align}\label{u-2}
\int_{\Omega}u^p \mathrm{d}x \leqslant C_{30},
\quad t \in (0,T_{\max}).
\end{align}
\end{lem}
\begin{proof}
Multiplying the first equation in (\ref{1.0.0}) by $u^{p-1}$($p>1$), and integrating by parts, followed by the application of Young's inequality, we obtain 
\begin{align}\label{3.2.3.1}
   \frac{\mathrm{d}}{\mathrm{d} t} \int_{\Omega}u^p \mathrm{d} x 
  = &-p(p-1) \int_{\Omega}u^{p-2} |\nabla u|^2 \mathrm{d} x
     -p(p-1) \int_{\Omega} u^{p-1} f(|\nabla v|^2) \nabla u \cdot \nabla v \mathrm{d}x \notag\\
  \leqslant & -\frac{2(p-1)}{p} \int_{\Omega}|\nabla u^{\frac{p}{2}}|^2 \mathrm{d}x
              +\frac{p(p-1)}{2}K^2_{1} \int_{\Omega} u^{p}\frac{|\nabla v|^2}{(1+|\nabla v|^2)^{2\alpha}} \mathrm{d}x
\end{align}
for all $t \in (0,T_{\max})$. Due to $\frac{x}{(1+x)^{2\alpha}}\leqslant 1$ by $\alpha \geqslant \frac{1}{2}$, we obtain 
\begin{align}\label{+u^p-large}
  \frac{\mathrm{d}}{\mathrm{d} t} \int_{\Omega}u^p \mathrm{d}x + \int_{\Omega} u^p \mathrm{d}x 
  \leqslant -\frac{2(p-1)}{p} \int_{\Omega}|\nabla u^{\frac{p}{2}}|^2 \mathrm{d}x 
            +\left(\frac{p(p-1)}{2}K^2_{1}+1\right) \int_{\Omega} u^p \mathrm{d}x
\end{align}
for all $t \in (0,T_{\max})$. By (\ref{u2-G-N}) and Young's inequality, there exists a constant $C_{31}=C_{31}(p,\alpha)>0$ such that
\begin{align*}
\Big(\frac{p(p-1)}{2}K^2_{1}+1\Big) \int_{\Omega} u^{p} \mathrm{d}x
  \leqslant & \frac{2(p-1)}{p} ||\nabla u^{\frac{p}{2}}||^{2}_{L^2(\Omega)} +C_{31}.
\end{align*}
Inserting this into (\ref{+u^p-large}), we have 
\begin{align}\label{reslut-1}
  \frac{\mathrm{d}}{\mathrm{d} t} \int_{\Omega}u^p \mathrm{d}x + \int_{\Omega} u^p \mathrm{d}x 
  \leqslant C_{31},
\end{align}
which implies (\ref{u-2}).
\end{proof}
Applying the standard Moser-type iterative argument, we finally obtain the $L^{\infty}$ estimate of $u$.\\
\emph{\textbf{Proof of Theorem \ref{0.3}.}} Due to $\frac{2k+4\alpha}{n(1-2\alpha)} \rightarrow \infty$ as $k\rightarrow \infty$, by Lemma~\ref{lem-3.2.3} and Lemma~\ref{lem-3.2.4}, for any $p>1$, we have
\begin{align*}
\sup _{0<t<T_{\max }}\|u(\cdot, t)\|_{L^{p}(\Omega)} \leqslant C_{32}.
\end{align*}
Using the standard Dirichlet heat semigroup estimates again, we fix $p>n$ to deduce that
\begin{align*}
\left\|\nabla v(\cdot, t)\right\|_{L^{\infty}(\Omega)} 
= &\|\nabla\left(v(\cdot, t)-M\right)\|_{L^{\infty}(\Omega)} \notag \\
= &\left\|\nabla e^{t \Delta}\left(v_0-M\right)-\int_0^t \nabla e^{(t-s) \Delta}\big(u(\cdot, s) v(\cdot, s)\big) \mathrm{d} s\right\|_{L^{\infty}(\Omega)} \notag \\
\leqslant & C_3 \left\|v_0-M\right\|_{W^{1, \infty}(\Omega)} \notag \\
           &+C_4 \|v_0\|_{L^{\infty}(\Omega)} \|u\|_{L^{p}(\Omega)} \int_0^{\infty}\left(1+t^{-\frac{1}{2}-\frac{n}{2p}}\right) e^{-\lambda t} \mathrm{d} t \notag \\
\leqslant & C_{33},
\quad t \in (0,T_{\max}).
\end{align*}
Thus, the statement of Theorem \ref{0.3} can be established through a Moser-type iterative argument (cf.  \cite[Lemma A.1]{2012-JDE-TaoWinkler}).
\hfill$\Box$\\
\section{Boundedness when $\tau=0$ and $\alpha > \frac{n-2}{2n}$ }\label{section 2}

In this section, we aim to prove the boundedness of solution when $\alpha > \frac{n-2}{2n}$ and $\tau=0$. Throughout this section, we assume that $\Omega$ is a general bounded domain in $\mathbb{R}^n$. We first give a basic observation on the regularity of signal gradient. 
\begin{lem}\label{lem-3.1.0}
Let $(u,v)$ be the classical solution of problem $(\ref{1.0.0})$ with $\tau=0$. Suppose that $f$ and $u_0$ satisfy $(\ref{f-lip})$ and $(\ref{u_0})$, then there exists a constant $C_{34}>0$ such that 
\begin{align}\label{nabla v}
  \int_{\Omega} |\nabla v(x,t)|^2 \mathrm{d}x \leqslant C_{34},
  \quad t \in (0,T_{\max}).
\end{align}
\end{lem}
\begin{proof}
Testing the second equation in (\ref{1.0.0}) by $(M-v)$ and integrating by parts, applying the fact that $v \leqslant M$ and (\ref{mass-0}), we see that 
\begin{align*}
  \int_{\Omega} |\nabla v|^2 \mathrm{d}x
  \leqslant \int_{\Omega} uv(M-v)
 \leqslant M^2 \int_{\Omega} u \mathrm{d}x
 =M^2 \int_{\Omega} u_0 \mathrm{d}x,
 \quad t \in (0,T_{\max}),
\end{align*}
which implies (\ref{nabla v}).
\end{proof}
The following lemma provides the $L^p$ estimate of $u$.
\begin{lem}\label{lem-3.1.1}
Let $(u,v)$ be the classical solution of problem $(\ref{1.0.0})$ with $\tau=0$. Suppose that $f(\xi)$ satisfies $(\ref{f-lip})$ and $(\ref{f-global})$ with $\alpha>\frac{n-2}{2n}$ and $K_{0}>0$. Then, for $p >2$ and $u_0$ fulfilling $(\ref{u_0})$, there exists a constant $C_{35}=C_{35}(\alpha,p,K_{0})>0$ such that 
\begin{align}\label{3.1.1.1}
  \int_{\Omega}u^p(x, t) \mathrm{d} x \leqslant C_{35},
   \quad t \in\left(0, T_{\max }\right).
\end{align}
\end{lem}
\begin{proof}

\textit{Case 1}: $\alpha \geqslant \frac{1}{2}$. Its proof is the same as Lemma~\ref{lem-3.2.4}.

\textit{Case 2}:  $\frac{n-2}{2n}<\alpha<\frac{1}{2}$ and $\tau=0$. Applying Hölder's inequality and Lemma~\ref{lem-3.1.0}, it follows from (\ref{3.2.3.1}) that
\begin{align}\label{+u^p-small-0}
  \frac{\mathrm{d}}{\mathrm{d} t} \int_{\Omega}u^p \mathrm{d}x 
  \leqslant & -\frac{2(p-1)}{p} \int_{\Omega}|\nabla u^{\frac{p}{2}}|^2 \mathrm{d}x 
            +\frac{p(p-1)}{2} K^2_{0} \int_{\Omega} u^p |\nabla v|^{2-4\alpha}\mathrm{d}x \notag \\
  \leqslant & -\frac{2(p-1)}{p} \int_{\Omega}|\nabla u^{\frac{p}{2}}|^2 \mathrm{d}x 
             +\frac{p(p-1)}{2} K^2_{0}\big(\int_{\Omega} u^{\frac{p}{2\alpha}} \mathrm{d}x \big)^{2\alpha} \big( \int_{\Omega}|\nabla v|^2 \mathrm{d}x \big)^{1-2\alpha} \notag\\
  \leqslant & -\frac{2(p-1)}{p} \int_{\Omega}|\nabla u^{\frac{p}{2}}|^2 \mathrm{d}x 
             +C_{34}^{1-2\alpha}\frac{p(p-1)}{2} K^2_{0} \big(\int_{\Omega} u^{\frac{p}{2\alpha}} \mathrm{d}x \big)^{2\alpha}
\end{align}
for all $t \in (0,T_{\max})$. Since $\alpha>\frac{n-2}{2n}$ implies $\frac{2n(p-2\alpha)}{pn+2-n}<2$, we use Young's inequality and Gagliardo-Nirenberg inequality to pick constants $C_{36}=C_{36}(p,\alpha,K_{0})>0$ and $C_{37}=C_{37}(p,\alpha,K_{0})>0$ such that
\begin{align}\label{G-N-1-1}
C_{34}^{1-2\alpha}\frac{p(p-1)}{2} K^2_{0}\Big( \int_{\Omega} u^{\frac{p}{2\alpha}} \mathrm{d}x \Big)^{2\alpha} 
= & C_{34}^{1-2\alpha}\frac{p(p-1)}{2} K^2_{0}\|u^{\frac{p}{2}}\|^{2}_{L^{\frac{1}{\alpha}}(\Omega)} \notag \\
  \leqslant & C_{36} \|\nabla u^{\frac{p}{2}}\|^{\frac{2n(p-2\alpha)}{pn+2-n}}_{L^2(\Omega)} \|u^{\frac{p}{2}}\|^{\frac{2(2-n+2n\alpha) }{pn+2-n}}_{L^{\frac{2}{p}}(\Omega)}
              +C_{36} \|u^{\frac{p}{2}}\|^{2}_{L^{\frac{2}{p}}(\Omega)} \notag\\
  \leqslant & \frac{p-1}{p} \|\nabla u^{\frac{p}{2}}\|^{2}_{L^2(\Omega)} +C_{37}.
\end{align}
Similarly, we have 
\begin{align}\label{G-N-1-2}
\int_{\Omega} u^p  \mathrm{d}x  \leqslant & \frac{p-1}{p} \|\nabla u^{\frac{p}{2}}\|^{2}_{L^2(\Omega)} +C_{37}.
\end{align}
Combining (\ref{+u^p-small-0}) with (\ref{G-N-1-1}) and (\ref{G-N-1-2}), we have
\begin{align}\label{reslut-3}
\frac{\mathrm{d}}{\mathrm{d} t} \int_{\Omega}u^p \mathrm{d}x + \int_{\Omega}u^p \mathrm{d}x \leqslant 2C_{37}
\end{align}
for all $t \in (0,T_{\max})$, which implies (\ref{3.1.1.1}).
\end{proof}

Applying the standard Moser-type iterative argument, we finally obtain the $L^{\infty}$ estimate of $u$.\\
\emph{\textbf{Proof of Theorem \ref{0.2}.}} For any $p \geqslant 2$, by Lemma~\ref{lem-3.1.1} and the standard elliptic regularity theory, we can find a constant $C_{38}(\alpha,p,K_{0})>0$ such that 
\begin{align*}
  \|\nabla v(\cdot, t)\|_{L^{\infty}(\Omega)} \leqslant C_{38}(\alpha,p,K_{0}), 
  \quad t \in\left(0, T_{\max }\right).
\end{align*}
Based on this and Lemma~\ref{lem-3.1.1}, the statement of Theorem \ref{0.2} can be established through a Moser-type iterative argument.
\hfill$\Box$\\
\vskip 3mm
\section{Blow-up when $\alpha <0$, $n=2$ and $\tau=0$}\label{section 1}
In this section, we aim to prove Theorem~\ref{0.1}, and assume that $\Omega =B_R(0) \subset \mathbb{R}^2$ and $\tau=0$. In the spirit of \cite{2023-CVPDE-AhnWinkler, 2023-PRSESA-WangWinkler, 2022-IUMJ-Winkler}, we introduce the mass distribution function 
\begin{align}\label{lem-2.1.2.1}
  w(s, t):=\int_0^{\sqrt{s}}\rho u \left(\rho, t \right) \mathrm{d}\rho, 
  \quad  (s,t) \in\left[0, R^2\right] \times \left[0, T_{\max }\right).
\end{align}
According to (\ref{1.0.0}) and (\ref{lem-2.1.2.1}), $w$ satisfies the following Dirichlet problem
\begin{align}\label{lem-2.1.2.3}
  \left\{
  \begin{array}{ll}
    w_t(s, t)=4 s w_{s s}(s, t)+ 2\sqrt{s} w_s(s, t) f\big(v_r^{2}\left(\sqrt{s}, t\right)\big) v_r\big(\sqrt{s}, t\big), 
    & s \in\left(0, R^2\right),\ t \in\left(0, T_{\max }\right), \\
    w(0, t)=0, \quad w\left(R^2, t\right)=\frac{m}{2\pi}, 
    & t \in\left[0, T_{\max }\right),\\
    w(s, 0)=w_0(s), 
    & s \in\left(0, R^2\right),
  \end{array}
  \right.
\end{align}
where $m:=\int_{\Omega} u_0 \mathrm{d} x$ and $w_0(s)=\int_0^{\sqrt{s}} \rho u_0(\rho) \mathrm{d} \rho$.
By Proposition~\ref{0.0-0}, we know that $w \in C^0\left(\left[0, T_{\max }\right) ; C^1\left(\left[0, R^2\right]\right) \cap C^{2,1}\left(\left(0, R^2\right] \times
\left(0, T_{\max }\right)\right)\right.$ satisfies
\begin{align}\label{lem-2.1.2.2}
  w_s(s, t)=\frac{1}{2} u\big(\sqrt{s}, t\big), 
  \quad (s,t) \in\left[0, R^2\right] \times \left[0, T_{\max }\right).
\end{align}

The following two lemmas introduce some useful estimates to derive the lower bound for $w_t$. The estimate of $r v_r(r,t)$ can be derived by a ODE comparison argument, referring to \cite[Lemma 3.1]{2023-PRSESA-WangWinkler} for details
of the proof. 
\begin{lem}\label{lem2.2.1} 
Suppose that $f$ and radially symmetric $u_0$ satisfy $(\ref{f-lip})$ and $(\ref{u_0})$, then 
\begin{align}\label{lem2.2.1.1}
  r v_r(r, t) \geqslant \frac{U(r, t) v(r, t)}{1+\int_0^r \frac{U \left(\rho, t\right)} {\rho} \mathrm{d} \rho}, 
  \quad (r,t) \in(0, R) \times \left(0, T_{\max }\right) \text {, }
\end{align}
where 
\begin{align*}
  U(r, t):=w \big(r^2, t\big), 
  \quad (r,t) \in[0, R] \times \left[0, T_{\max }\right).
\end{align*}
\end{lem}
Considering Lemma~\ref{lem2.2.1}, we estimate $v$ from below which detailed proof can be found in \cite[ Lemma 2.4]{2023-CVPDE-AhnWinkler}.

\begin{lem}\label{lem-2.2.2}
Suppose that $f$ and radially symmetric $u_0$ satisfy $(\ref{f-lip})$ and $(\ref{u_0})$, then
\begin{align}\label{lem-2.2.2.1}
  v(r, t) \geqslant M \exp \Big[-\big(\frac{m}{2 \pi} \cdot \ln \frac{R}{r}\big)^{\frac{1}{2}}\Big], 
  \quad (r,t) \in (0, R] \times \left(0, T_ {\max }\right).
\end{align}
\end{lem}
The following estimate of $w_t$ comes from the above lemmas. 
\begin{lem}\label{lem2.2.3}
Suppose that radially symmetric $u_0$ satisfies $(\ref{u_0})$. Assume that $f(\xi)$ satisfies $(\ref{f-lip})$ and $(\ref{f-blow-up})$ with  $\alpha<0$ and $k>0$. Then, for any $\delta>0$, there exists a constant $C_{39}=C_{39}(\delta, R) >0$, such that
\begin{align}\label{wt}
  w_t(s,t) \geqslant 4sw_{s s}(s,t) 
  + M^{1-2\alpha} kC_{39} 
  \cdot s^{\delta(1-2\alpha)+\alpha} \frac{ w^{1-2\alpha}(s,t) w_s(s,t) }{\big({1+\frac{1}{2} \int_0^{s} \frac{ w(\sigma, t)}{\sigma} \mathrm{d} \sigma}\big)^{1-2\alpha}}
\end{align}
for all $s \in \left(0, R^2\right)$ and $t \in \left(0, T_{\max }\right)$.
\end{lem}
\begin{proof}
By means of Young's inequality, for any $\delta>0$, we have
\begin{align*}
  \left(\frac{m}{2 \pi} \cdot \ln \frac{R}{\sqrt{s}}\right)^{\frac{1}{2}} 
  \leqslant & 2\delta \ln \frac{R}{\sqrt{s}}+\frac{m}{16\pi \delta} \notag\\
  = &\ln \frac{R^{2\delta}}{s^{\delta}}+\frac{m}{16\pi\delta}.
\end{align*}
Inserting this into (\ref{lem-2.2.2.1}), we deduce that
\begin{align}\label{lem-2.2.2.2}
  v(\sqrt{s}, t) \geqslant M R^{-2\delta} e^{-\frac{m}{16 \pi \delta}} s^{\delta},  
  \quad (s,t) \in \left(0, R^2\right) \times \left(0, T_{\max }\right) \text {. }
\end{align}
According to the definitions of $U$ and $w$, we obtain
\begin{align}\label{lem2.2.3.3}
  \int_0^{\sqrt s} \frac{U(\rho, t)} {\rho} \mathrm{d} \rho
  =\frac{1}{2} \int_0^s \frac{w(\sigma, t)}{\sigma} \mathrm{d} \sigma.
\end{align}
By Lemma~\ref{lem2.2.1}, (\ref{lem-2.2.2.2}), (\ref{lem2.2.3.3}) and the positivity of $w$, we derive that
\begin{align}\label{lem2.2.3.2}
  v_r \big(\sqrt {s} ,t\big) 
  \geqslant &M R^{-2\delta} e^{-\frac{m}{16 \pi \delta}} \frac{s^{\delta-\frac{1}{2}} w(s, t) }{1+\frac{1}{2} 
    \int_0^s \frac{w(\sigma, t)}{\sigma} \mathrm{d} \sigma},
    \quad (s,t) \in \left(0, R^2\right) \times \left(0, T_{\max }\right) \text {. }
\end{align} 
Using (\ref{f-blow-up}), and inserting (\ref{lem2.2.3.2}) into (\ref{lem-2.1.2.3}), there exists a constant $C_{39}(m,R,\delta)=2R^{-2\delta(1-2\alpha)}e^{-\frac{m(1-2\alpha)}{16\pi \delta}}>0$ such that
\begin{align}\label{w-n}
  w_t(s, t) 
  \geqslant &4 s w_{s s}(s, t)+ 2k\sqrt{s} w_s(s, t) \big(1+v_r^2(\sqrt{s}, t)\big)^{-\alpha}v_r(\sqrt{s}, t) \notag\\
  \geqslant &4 s w_{s s}(s, t)+ 2k\sqrt{s} w_s(s, t) v_r^{1-2\alpha}\big(\sqrt{s}, t\big) \notag\\
  \geqslant &4sw_{s s}(s,t) + M^{1-2\alpha} k_fC_{39}  \cdot \frac{ s^{\delta(1-2\alpha)+\alpha} w^{1-2\alpha}(s,t) w_s(s,t) }{\big({1+\frac{1}{2} \int_0^{s} \frac{ w(\sigma, t)}{\sigma} \mathrm{d} \sigma}\big)^{1-2\alpha}}
\end{align}
for all $s \in \left(0, R^2\right)$ and $t \in \left(0, T_{\max }\right)$. We complete our proof.
\end{proof}
The method of detecting blow-up depends on a differential inequality of a moment-type functional $\phi(t)$. For any given $\gamma=\gamma\left(\alpha\right) \in(0,1)$, we define such a positive functional
\begin{align}\label{2.3.0.1}
  \phi(t):=\int_0^{R^2} s^{-\gamma} w(s, t) \mathrm{d} s, 
  \quad t \in\left[0, T_{\max }\right),
\end{align}
which is well-defined and belongs to $C^0\left(\left[0, T_{\max }\right)\right) \cap C^1\left(\left(0, T_{\max }\right)\right)$. To prepare our subsequent analysis of $\phi(t)$, given $\gamma \in (0,1)$, we further introduce an auxiliary functional 
\begin{align}\label{2.3.0.2}
  \psi(t):=\int_0^{R^2}  s^{\delta(1-2\alpha)+\alpha-\gamma-1} w^{2-2\alpha}(s,t)   \mathrm{d}s, 
  \quad t \in\left[0, T_{\max }\right)
\end{align} 
 $\psi \in C^0 \left[ 0,T_{max} \right)$.

In the following lemma, we establish a basic differential inequality of the moment-type functional $\phi(t)$.
\begin{lem}\label{lem-2.3.2}
Suppose that radially symmetric $u_0$ satisfies $(\ref{u_0})$. Assume that $f(\xi)$ satisfies $(\ref{f-lip})$ and $(\ref{f-blow-up})$ with $\alpha<0$ and $k>0$. Then, for any given $\delta=\delta(\alpha) \in \big(0,\frac{-\alpha}{1-2\alpha}\big)$ and $\gamma=\gamma(\delta) \in (0,1)$, there exists a constant $C_{40}=C_{40}(\delta, R)>0$, such that the functions $\phi(t)$ and $\psi(t)$ satisfy
\begin{align}\label{lem-2.3.2.1}
  \phi^{\prime}(t) 
  \geqslant \frac{M^{1-2\alpha}}{C_{40}} \cdot \frac{\psi(t)}{1+\psi^{\frac{1-2\alpha}{2-2\alpha}}(t)}
       -C_{40} \psi^{\frac{1}{2-2\alpha}}(t) -C_{40},
  \quad t\in (0,T_{\max})
\end{align}
and 
\begin{align}\label{lem-2.3.2.2}
  \phi(t) \leqslant C_{40} \psi^{\frac{1}{2-2\alpha}}(t),
  \quad t \in\left(0, T_{\max }\right).
\end{align}
\end{lem}
\begin{proof}
Due to $0<\delta<\frac{-\alpha}{1-2\alpha}$, we deduce that
\begin{align*}
 \delta(1-2\alpha)+\alpha -\frac{-\alpha-\delta(1-2\alpha)}{1-2\alpha} =(\delta(1-2\alpha)+\alpha)\frac{2-2\alpha}{1-2\alpha}<0.
\end{align*}
By $\alpha<0$, we have
\begin{align*}
 \delta(1-2\alpha)+\alpha < \delta(1-2\alpha)+2-\alpha.
\end{align*}
Thus, we can fix $\gamma=\gamma(\alpha) \in (0,1)$ fulfilling 
\begin{align}\label{ga}
  \delta(1-2\alpha)+\alpha<\gamma< \min\left\{\frac{-\alpha-\delta(1-2\alpha)}{1-2\alpha},\delta(1-2\alpha)+2-\alpha\right\}.
\end{align}
According to $\gamma< \frac{-\alpha-\delta(1-2\alpha)}{1-2\alpha}$ in (\ref{ga}) and $\alpha<0$, we find that
\begin{align}\label{>-1-1}
 &\frac{-\alpha+\gamma+1-\delta(1-2\alpha)-\gamma(2-2\alpha)}{1-2\alpha} \notag \\
 >&\frac{-\alpha+\gamma+1-\delta(1-2\alpha)-(\gamma+1)(2-2\alpha)}{1-2\alpha} \notag \\
 = &\frac{-\alpha}{1-2\alpha}-(\gamma-1+\delta) \notag \\
 >&-1.
\end{align}
Noting that the assumption $\gamma>\delta(1-2\alpha)+\alpha$ in (\ref{ga}) warrants that 
\begin{align}\label{<-1-2}
\frac{-\alpha+\gamma+1-\delta(1-2\alpha)-(2-2\alpha)}{1-2\alpha}-(-1)=-\alpha+\gamma-\delta(1-2\alpha)>0.
\end{align}
Inserting (\ref{wt}) into the definition of $\phi(t)$ as (\ref{2.3.0.1}), we derive that
\begin{align}\label{lem-2.3.2.3}
  \phi^{\prime}(t) 
  \geqslant 4\int_0^{R^2}s^{1-\gamma}w_{s s}(s,t) + M^{1-2\alpha} k_fC_{39}  \cdot 
  \int_0^{R^2} \frac{ s^{\delta(1-2\alpha)+\alpha-\gamma} w^{1-2\alpha}(s,t) w_s(s,t) }{\big({1+\frac{1}{2} \int_0^{s} \frac{ w(\sigma, t)}{\sigma} \mathrm{d} \sigma}\big)^{1-2\alpha}}
\end{align}
for all $ t \in\left(0, T_{\max }\right)$. Due to $\gamma<1$ and $\gamma<\delta(1-2\alpha)+2-\alpha$, we obtain
\begin{align}\label{lem-2.3.2.12}
  s^{1-\gamma}w_s(s,t)\rightarrow 0, \quad  s^{-\gamma}w(s,t)\rightarrow 0 \  \text{ and } \ s^{\delta(1-2\alpha)+\alpha-\gamma}w^{2-2\alpha}(s,t)\rightarrow 0   \ \text{as}\
  s\rightarrow 0
\end{align}
for all $ t \in (0,T_{\max})$.

In the following, we first deal with the first term on the right side of (\ref{lem-2.3.2.3}). We integrate by parts and combine (\ref{lem-2.3.2.12}) with $w_s(s,t)\geqslant 0$ inferred from (\ref{lem-2.1.2.2}) to see that
\begin{align}\label{lem-2.3.2.7}
  4 {\int_0^{R^{2}}}s^{1-\gamma}  w_{s s}(s, t)\mathrm{d}s 
  = &-4(1-\gamma) \int_0^{R^{2}} s^{-\gamma}w_s(s,t) \mathrm{d}s 
      +4R^{2 (1-\gamma)} w_s(R^2,t) \notag\\
  \geqslant &-4(1-\gamma) \int_0^{R^{2}} s^{-\gamma}w_s(s,t) \mathrm{d}s \notag\\
  = &-4\gamma (1-\gamma) \int_0^{R^2} s^{-\gamma-1}w(s,t) \mathrm{d}s 
     -4(1-\gamma)\frac{R^{-2\gamma}m}{2 \pi}.
\end{align}
Due to the definition of $\psi(t)$, applying Young's inequality, we find that
\begin{align}\label{lem-2.3.2.8}
 & \int_0^{R^2}s^{-\gamma-1}w(s,t) \mathrm{d}s \notag \\
   \leqslant & \left(\int_0^{R^2} s^{\delta(1-2\alpha)+\alpha-\gamma-1} w^{2-2\alpha}(s,t) \mathrm{d}s \right)^{\frac{1}{2-2\alpha}} 
   \left(\int_0^{R^2} s^{\frac{-\alpha+\gamma+1-\delta(1-2\alpha)-(\gamma+1)(2-2\alpha)}{1-2\alpha}} \mathrm{d}s \right)^{\frac{1-2\alpha}{2-2\alpha}} \notag\\
   \leqslant &  C_{41}\psi^{\frac{1}{2-2\alpha}}(t),
\end{align}
where $ C_{41}=C_{41}(\alpha, R)=\left(\int_0^{R^2} s^{\frac{-\alpha+\gamma+1-\delta(1-2\alpha)-(\gamma+1)(2-2\alpha)}{1-2\alpha}} \mathrm{d}s \right)^{\frac{1-2\alpha}{2-2\alpha}}$ is finite because of (\ref{>-1-1}). Substituting (\ref{lem-2.3.2.8}) into (\ref{lem-2.3.2.7}), we obtain
\begin{align}\label{lem-2.3.2.9}
  4{\int_0^{R^{2}}}s^{1-\gamma}  w_{s s}(s, t)\mathrm{d}s 
  \geqslant -4\gamma (1-\gamma)C_{41} 
       \psi^{\frac{1}{2-2\alpha}}(t)-4(1-\gamma)\frac{R^{-2\gamma}m}{2 \pi}.
\end{align}

Next, we deal with the second term on the right side of (\ref{lem-2.3.2.3}). Using Hölder's inequality, we have
\begin{align}\label{lem-2.3.2.10}
  \int_0^s \frac{w(\sigma, t)}{\sigma} \mathrm{d} \sigma 
  \leqslant & \left(\int_0^s \sigma^{\frac{-\alpha+\gamma+1-\delta(1-2\alpha)-(2-2\alpha)}{1-2\alpha}}\mathrm{d}\sigma\right)^{\frac{1-2\alpha}{2-2\alpha}} \cdot 
         \left(\int_0^s \sigma^{\delta(1-2\alpha)+\alpha-\gamma-1} w^{2-2\alpha}(s,t) \mathrm{d}\sigma\right)^ {\frac{1}{2-2\alpha}} \notag\\
  \leqslant & {C_{42}}\psi^{\frac{1}{2-2\alpha}}(t), 
\end{align}
where $C_{42}=C_{42}(\alpha,R)= \left(\int_0^{R^2} \sigma^{\frac{-\alpha+\gamma+1-\delta(1-2\alpha)-(2-2\alpha)}{1-2\alpha}} \mathrm{d}\sigma\right)^{\frac{1-2\alpha}{2-2\alpha}}>0$ is finite because of (\ref{<-1-2}). Integrating by parts, and using (\ref{lem-2.3.2.12}), we obtain
\begin{align}\label{lem-2.4-2.23}
  \int_0^{R^2} s^{\delta(1-2\alpha)+\alpha-\gamma} w^{1-2\alpha}(s,t) w_s(s,t) \mathrm{d}s 
  = & -\frac{\delta(1-2\alpha)+\alpha-\gamma}{2-2\alpha} \int_0^{R^2} s^{\delta(1-2\alpha)+\alpha-\gamma-1} w^{2-2\alpha} \mathrm{d}s \notag\\
    & + \frac{R^{2\delta(1-2\alpha)+2\alpha-2\gamma}}{2-2\alpha} \left(\frac{m}{2\pi}\right)^{2-2\alpha} \notag\\
  \geqslant &C_{43}  \psi(t),
\end{align}
where $C_{43}=C_{43}(\alpha,R)=-\frac{\delta(1-2\alpha)+\alpha-\gamma}{2-2\alpha}>0$ due to $\gamma>\delta(1-2\alpha)+\alpha$ and $\alpha<0$. As a consequence of (\ref{lem-2.3.2.10}) and (\ref{lem-2.4-2.23}), we have, 
\begin{align}\label{lem-2.3.2.11}
  \int_0^{R^2} \frac{ s^{\delta(1-2\alpha)+\alpha-\gamma} w^{1-2\alpha}(s,t) w_s(s,t) }{\big({1+\frac{1}{2} \int_0^{s} \frac{ w(\sigma, t)}{\sigma} \mathrm{d} \sigma}\big)^{1-2\alpha}}
  \geqslant &\frac{2^{1-2\alpha} C_{43} \psi(t)}{2^{1-4\alpha}+{2^{-2\alpha}C_{42}^{1-2\alpha}}\psi^{\frac{1-2\alpha}{2-2\alpha}}(t)} \notag \\
  \geqslant &\frac{2 C_{43}}{\max\left\{2^{1-2\alpha},{C_{42}^{1-2\alpha}}\right\}} \cdot \frac{\psi(t)}{1+\psi^{\frac{1-2\alpha}{2-2\alpha}}(t)}.
\end{align}
Therefore, (\ref{lem-2.3.2.1}) can be deduced from (\ref{lem-2.3.2.3}), (\ref{lem-2.3.2.9}) and (\ref{lem-2.3.2.11}). 

Applying Young's inequality and (\ref{>-1-1}), there exists a positive constant $C_{40}=C_{40}(\alpha,R)$ large enough such that
\begin{align*}
  \phi(t)=&\int_0^{R^2} s^{-\gamma} w(s, t) \mathrm{d} s \notag \\
  \leqslant & \left(\int_0^{R^2} s^{\alpha-\gamma-1+\delta(1-2\alpha)}w^{2-2\alpha} \mathrm{d}s \right)^{\frac{1}{2-2\alpha}} \cdot 
  \left(\int_0^{R^2} s^{\frac{-\alpha+\gamma+1-\delta(1-2\alpha)-\gamma(2-2\alpha)}{1-2\alpha}} \right)^{\frac{1-2\alpha}{2-2\alpha}} \notag\\
  \leqslant & C_{40} \psi^{\frac{1}{2-2\alpha}}(t),
  \quad t \in (0,T_{\max}),
\end{align*}
which implies (\ref{lem-2.3.2.2}). 
\end{proof}

\emph{\textbf{Proof of Theorem \ref{0.1}.}}
We denote
\begin{align}\label{lem-2.4.1.2}
  \phi_0=\phi_0\left(u_0\right):=\int_0^{R^2} s^{-\gamma} w_0(s) \mathrm{d} s
\end{align}
and
\begin{align}\label{lem-2.4.1.3}
  S:=\left\{t \in\left(0, T_{\max }\right) \mid \phi(t)>\frac{\phi_0}{2} \text { on }(0, t)\right\},
\end{align}
where $w_0(s)$ is defined in (\ref{lem-2.1.2.3}). We note that $S$ is not empty due to the continuity of $\phi(t)$, (\ref{2.3.0.1}) and (\ref{lem-2.4.1.2}). Thus, $T:=\sup S \in\left(0, T_{\max }\right] \subset(0, \infty]$ is well-defined. 


We first prove $\phi^{\prime}(t) \geqslant 0$ for all $t \in(0,T)$.
Writing
\begin{align*}
  f_M(z):=\frac{M^{1-2\alpha}}{2 C_{40}} \cdot \frac{z}{1+z^{\frac{1-2\alpha}{2-2\alpha}}}-C_{40} z^{\frac{1}{2-2\alpha}}-C_{40}
\end{align*}
and
\begin{align*}
C_{44}:=\left(\frac{\phi_0}{2 C_{40}}\right)^{2-2\alpha},
\end{align*}
then we have
\begin{align*}
  \inf _{z \geqslant C_{44}} f_M(z) \rightarrow+ \infty \text { \ as } \ M \rightarrow \infty.
\end{align*}
Due to (\ref{lem-2.3.2.2}) and (\ref{lem-2.4.1.3}), we have
\begin{align*}
  \psi(t) 
  \geqslant \left(\frac{\phi(t)}{C_{40}}\right)^{2-2\alpha} 
  \geqslant \left(\frac{\phi_0}{2 C_{40}}\right)^{2-2\alpha}, 
  \quad t \in(0,T).
\end{align*}
Thus, there exists a constant $M^{\star}(u_0)>0$ such that
\begin{align}\label{f_M}
  f_M(\psi(t)) \geqslant 0, \quad  t \in(0, T), \ M \geqslant M^{\star}(u_0).
\end{align}
We choose $ M \geqslant M^{\star}(u_0)$ in the following proof. The non-decreasing of $g(z):=\frac{z^{\frac{1-2\alpha}{2-2\alpha}}}{1+z^{\frac{1-2\alpha}{2-2\alpha}}}$ on $z\in (0,+\infty)$, together with (\ref{lem-2.3.2.1}), (\ref{lem-2.3.2.2}) and (\ref{f_M}) warrants that
\begin{align}\label{lem-2.4.1.6}
  \phi^{\prime}(t) 
  \geqslant &\frac{M^{1-2\alpha}}{2 C_{40}} \cdot \frac{\psi(t)}{1+\psi^{\frac{1-2\alpha}{2-2\alpha}}(t)}+f_M(\psi(t)) \notag\\
  \geqslant &\frac{M^{1-2\alpha}}{2 C_{40}} \cdot \frac{\psi^{\frac{1-2\alpha}{2-2\alpha}}(t)}{1+\psi^{\frac{1-2\alpha}{2-2\alpha}}(t)} {\psi}^{\frac{1}{2-2\alpha}} \notag\\
  \geqslant &\frac{M^{1-2\alpha}}{2 C_{40}} \cdot \frac{\left(\frac{\phi_0}{2 C_{40}}\right)^{1-2\alpha}}{1+\left(\frac{\phi_0}{2 C_{40}}\right)^{1-2\alpha}} \cdot \frac{\phi(t)}{C_{40}} \notag\\
  = &C_{45} \phi(t), 
  \quad  t \in(0, T),
\end{align} 
where $C_{45}=C_{45}\left(\alpha, R, M\right)=\frac{M^{1-2\alpha}}{2 C_{40}^2} \cdot \frac{\left(\frac{\phi_0}{2 C_{40}}\right)^{1-2\alpha}}{1+\left(\frac{\phi_0}{2 C_{40}}\right)^{1-2\alpha}}$. By continuity of $\phi(t)$ and the definition of $T$, we have $\phi(T)=\frac{\phi_0}{2}$ which is evidently incompatible with the nondecrease of $\phi(t)$ on $[0,T)$ asserted by (\ref{lem-2.4.1.6}). Thus, we deduce that $T=T_{\max}$. 

It follows from (\ref{lem-2.4.1.6}) that
\begin{align*}
  \phi(t) \geqslant \phi_0 e^{C_{45}t}, \quad t \in(0,T).
\end{align*}
According to (\ref{2.3.0.1}) with $\gamma<1$ and the non-decreasing of $w$, we infer that
\begin{align*}
  \frac{m R^{2(1-\gamma)}}{2 \pi(1-\gamma)} 
  \geqslant \phi(t) 
  \geqslant \phi_0 e^{C_{45} t}, \quad t \in\left(0, T_{\max }\right),
\end{align*}
which implies 
\begin{align*}
  T_{\max } \leqslant \frac{1}{C_{45}} \ln \frac{m R^{2(1-\gamma)}}{2 \pi(1-\gamma) \phi_0}.
\end{align*}
We complete our proof.
\hfill$\Box$




\end{document}